%% file: PaperQuadraticConvergence_Arxiv.tex
\documentclass[a4paper,10pt]{article}

\usepackage{amsfonts}
\usepackage{amsmath}
\usepackage{amssymb}
\usepackage{amsthm}
\usepackage{algorithm}
\usepackage{algorithmic}
\usepackage{arydshln} % for dashed lines
\usepackage[english]{babel}
\usepackage[T1]{fontenc}
\usepackage[utf8]{inputenc}
\usepackage{multirow}
\usepackage{nicefrac}
\usepackage{subfigure}
\usepackage{todonotes}
\usepackage{url}
\usepackage{booktabs}

\theoremstyle{plain}% Theorem-like structures provided by amsthm.sty
\newtheorem{theorem}{Theorem}[section]
\newtheorem{lemma}[theorem]{Lemma}
\newtheorem{corollary}[theorem]{Corollary}

\newtheorem{assumption}[theorem]{Assumption}
\newtheorem*{assumptionGRA}{General Regularity Assumptions}
\newtheorem*{theoremQuad}{Quadratic Convergence Theorem}

\theoremstyle{definition}
\newtheorem{definition}[theorem]{Definition}
\newtheorem{example}[theorem]{Example}

\theoremstyle{remark}

\def\algname{\textup{\texttt{quADAPT}}}

\input{commands}

\title{An adaptive discretization method solving semi-infinite optimization problems with quadratic rate of convergence}

\author{Tobias Seidel\footnote{
		Corresponding author: Tobias Seidel. Email: tobias.seidel@itwm.fraunhofer.de}, Karl-Heinz K\"{u}fer
	\\
	Fraunhofer Institute for Industrial Mathematics (ITWM), \\
	Kaiserslautern, Germany \\
}

\begin{document}

% interactnlmsample.tex
% v1.05 - August 2017

%\documentclass[]{interact}
%
%\usepackage{epstopdf}% To incorporate .eps illustrations using PDFLaTeX, etc.
%\usepackage[caption=false]{subfig}% Support for small, `sub' figures and tables
%%\usepackage[nolists,tablesfirst]{endfloat}% To `separate' figures and tables from text if required
%%\usepackage[doublespacing]{setspace}% To produce a `double spaced' document if required
%%\setlength\parindent{24pt}% To increase paragraph indentation when line spacing is doubled
%\usepackage[numbers,sort&compress]{natbib}% Citation support using natbib.sty
%\bibpunct[, ]{[}{]}{,}{n}{,}{,}% Citation support using natbib.sty
%\renewcommand\bibfont{\fontsize{10}{12}\selectfont}% Bibliography support using natbib.sty
%\makeatletter% @ becomes a letter
%\def\NAT@def@citea{\def\@citea{\NAT@separator}}% Suppress spaces between citations using natbib.sty
%\makeatother% @ becomes a symbol again
%
%\usepackage[]{todonotes}
%\usepackage{algorithm}
%\usepackage{algorithmic}

\maketitle

\begin{abstract}
Semi-infinite programming can be used to model a large variety of complex optimization problems. The simple description of such problems comes at a price: semi-infinite problems are often harder to solve than finite nonlinear problems. In this paper we combine a classical adaptive discretization method developed by Blankenship and Falk in \cite{Blankenship.1976} and techniques regarding a semi-infinite optimization problem as a bi-level optimization problem. We develop a new adaptive discretization method which combines the advantages of both techniques and exhibits a quadratic rate of convergence. We further show that a limit of the iterates is a stationary point, if the iterates are stationary points of the approximate problems.
\end{abstract}


\section{Introduction}
In this paper we consider semi-infinite optimization problems. Therefore, Let $\UpLevIneqIndexset, \LowLevIneqIndexset$ be finite index sets and
\begin{align*}
\ObjFun:&\ \UpLevSpace\rightarrow\RNumbers \Komma\\
\UpLevIneq:&\ \UpLevSpace\times\LowLevSpace \rightarrow \UpLevIneqSpace \Komma\\
\LowLevIneq:&\ \LowLevSpace\rightarrow\LowLevIneqSpace
\end{align*}
be twice continuously differentiable functions. We consider throughout this work the following optimization problem:
\begin{align}
\nonumber\SIP : \min_{\UpLevVars\in \UpLevSpace}\quad& \ObjFun(\UpLevVars)\\
\text{s.t.}\quad & \NonLinIneqSingle(\UpLevVars,\LowLevVars) \leq 0 \text{ for all } \UpLevIneqIndexForAll, \LowLevVars\in \Indexset \Komma\label{eq:SemiConstrs}
\end{align}
where
\begin{equation*}
\Indexset=\left\{\LowLevVars \in \LowLevSpace \mid \LowLevIneq(\LowLevVars) \leq 0 
\right\} 
\end{equation*}
denotes the so-called \emph{semi-infinite indexset}. The inequalities in \eqref{eq:SemiConstrs} are called \emph{semi-infinite constraints}. We denote the feasible set of problem \SIP\ by $\FeasSet$. An overview over the theory and numerical methods for semi-infinite optimization can be found in the overview articles \cite{Hettich.1993,Lopez.2007} and the books \cite{Hettich.1982,Polak.1997}.

For every $\UpLevIneqIndexForAll$ and $\UpLevVars\in \UpLevSpace$ the $\UpLevIneqIndex$-th lower-level problem is denoted by
\begin{align*}
\LowLevProblem{\UpLevIneqIndex}{\UpLevVars}: \max_{\LowLevVars\in \LowLevSpace} \quad &\UpLevIneqSingle(\UpLevVars,\LowLevVars)\\
s.t.\quad& \LowLevIneqSingle(\LowLevVars)\leq 0 \text{ for all }\nonumber \LowLevIneqIndexForAll \Punkt
\end{align*}
By introducing, for every $\UpLevIneqIndexForAll$, the so-called optimal value function $\varphi_{\UpLevIneqIndex}(\UpLevVars)= \max_{\LowLevVars\in \Indexset} \UpLevIneqSingle(\UpLevVars,\LowLevVars)$, the feasible set can be described equivalently by
\begin{equation*}
\FeasSet= \{\UpLevVars \in \UpLevSpace\mid \varphi_{\UpLevIneqIndex}(\UpLevVars)\leq 0 \text{ for all } \UpLevIneqIndexForAll\} \Punkt
\end{equation*}
This bi-level structure is important in the investigation of the structural properties and the development of algorithms \cite{Stein.2003}. Using reformulations of convex lower-level problems, algorithms have been developed. In \cite{Stein.2003b} and \cite{Stein.2010} the authors used the KKT conditions  to replace the lower-level problem. To avoid complementarity conditions the authors in \cite{Diehl.2013} used the lower-level Wolfe duality.

In this paper we discuss the solution of semi-infinite problem with the help of an adaptive discretization methods. To do so we will not solve the original problem, but a series of finite nonlinear optimization problems which approximate the semi-infinite problem.

We consider a finite subset $\dot{Y} \subseteq \Indexset$. The problem
\begin{align*}
\SIP(\dot{Y}) : \min_{\UpLevVars\in\UpLevSpace}\quad& \ObjFun(\UpLevVars)\\
\text{s.t.}\quad& g(\UpLevVars,\LowLevVars) \leq 0 \text{ for all } \UpLevIneqIndexForAll, \LowLevVars \in \dot{Y} \Punkt
\end{align*}
is called \emph{discretized problem}. An overview over classical discretization methods can be found in \cite{Reemtsen.1991,Reemtsen.1998b}.

An algorithm which adaptively chooses the discretization points has already been introduced by Blankenship and Falk in \cite{Blankenship.1976}. The algorithm consists in every iteration $\ItIndex$ of two steps:
\begin{itemize}
	\item (Optimization Step) Determine for a current discretization $\ItIndexset$ a solution $\ItUpLev$ of the discretized problem $\DiscSIP{\ItIndexset}$.
	\item (Refinement Step) For every $\UpLevIneqIndexForAll$ determine the index $\ItLowLev$ which  is violated most, i.e. 
	\begin{equation*}
		\UpLevIneqSingle(\ItUpLev, \ItLowLev)=\max_{\LowLevVars\in \Indexset} \UpLevIneqSingle(\ItUpLev,\LowLevVars)
	\end{equation*}
	and add these indices to the discretization.
\end{itemize}
This simple scheme has been revisited multiple times in the literature and modified versions have been published. Reemtsen used in \cite{Reemtsen.1991} a fine discretization and solved the lower-level on this fine discretization. In \cite{Tsoukalas.2011} the authors modified the algorithm by Blankenship and Falk to obtain feasible points after finitely many steps. Mitsos then used these ideas in \cite{Mitsos.2011} to obtain an outer and an inner approximation for global optimization. In \cite{Mitsos.2015} and \cite{Djelassi.2019} the ideas are extended to generalized semi-infinite optimization using a reformulation with disjunctive constraints. In his PhD-thesis \cite{Schwientek.2013} Schwientek used a transformation function to solve generalized semi-infinite optimization problems more directly with the Blankenship and Falk algorithm.

Typical convergence results include statements about properties of an accumulation point of the iterates $\ItUpLev$.
For the original algorithm they can be found for example in \cite{Blankenship.1976} and \cite{Reemtsen.1994}. In \cite{Still.2001b} Still investigated the rate of convergence for a classical discretization methods. He bounded the distance to an optimal solution of \SIP\ in terms of the mesh-size. To the best of our knowledge there is no work investigating the rate of convergence of the iterates constructed according to the Blankenship and Falk algorithm.

In the following we will present and example which shows that in general no quadratic rate of convergence will hold. The reason lies in the fact, that only the points in the discretization are considered in the Optimization Step. All other points in the index-set are not considered. For the 'final' steps towards a higher accuracy many iterations are needed. Not only the number of iterations grows, but also the number of discretization points added. Which increases the time needed for every iteration. This can make it hard to solve applications to a high precision.

To overcome this slow convergence, we propose a new adaptive discretization method having a quadratic rate of convergence. To do this, we consider the Refinement Step in which the lower-level problem is solved. Instead of only adding a global solution to the discretization we also calculate derivatives that describe the dependence of the solution on the optimization variables. Using this linear information,  we construct an additional constraint that takes into account the points that have not yet been added to the discretization. While the discretization points guarantee convergence towards feasibility, the new constraints guarantee a quadratic rate of convergence.

We explain the new approach in more detail and introduce the algorithm in Section \ref{sec:Method}. In the following sections we investigate the new method. In Section \ref{sec:Stationary} we show that if in every iteration a KKT point is calculated, than a limit point is a stationary point of the semi-infinite problem. We next assume that this limit point is strongly stable and show in Section \ref{sec:Quadratic} that the iterates converge quadratically. In Section \ref{sec:Numeric} we investigate the quadratic rate of convergence with a numerical example.

\section{A new adaptive discretization method: \algname}\label{sec:Method}
We begin this chapter with an example which shows that in general no quadratic rate of convergence can be expected for the adaptive discretization by Blankenship and Falk. As mentioned in the introduction, we want to add more information about the lower level problems in the discretized problems. An important technique will be the well known concept of the so-called Reduction Ansatz (see for example \cite{Hettich.1982,Jongen.1992,Klatte.1992}) We introduce this technique in Subsection \ref{subsec:Reduction}. We then show how the Reduction Ansatz can be used to also account for all points in $\Indexset$ which are not yet added the discretization. This considerations lead to Algorithm \algname. We end this section by revisiting Example \ref{ex:LinearConvergence} and show that the new algorithm improves the rate of convergence.

\begin{example} \label{ex:LinearConvergence}
	We consider the following semi-infinite optimization problem:
	\begin{align*}
	\SIPEx: \min_{\UpLevVars\in \UpLevSpace} \quad &-x_1+\frac{3}{2}x_2\\
	s.t. \quad&-\LowLevVarsSingle^2 +2\LowLevVarsSingle\cdot x_1-x_2 \leq 0 \quad \text{ for all } \enspace \LowLevVarsSingle\in[-1,1] \Komma\\
	&x_1,x_2 \in [-1,1] \Punkt
	\end{align*}
	The solution of the lower level problem is given by by $\LowLevVarsSingle=x_1$.
	This means that the feasible set is given by $M=\{\UpLevVars \in [-1,1]^2\mid x_2\geq x_1^2\}$.
	Using the KKT conditions the global solution can be calculated and is given by:
	\begin{equation*}
	\LimUpLev=\left(\frac{1}{3},\frac{1}{9}\right)\Punkt
	\end{equation*}
	To see that in this example, we do not have a quadratic rate of convergence, we have to investigate the sequence of iterates constructed according to the Blankenship and Falk algorithm a bit further. Starting with an empty discretization the first iterates and corresponding solution of the lower level problems are given by
	\begin{equation*}
	\ItUpLev[0]= (1,-1)^{\top},\LowLevVarsSingle^0= 1, \quad \ItUpLev[1]= (0,-1)^{\top},	\LowLevVarsSingle^1= 0\Punkt
	\end{equation*}
	Using the KKT conditions for the discretized problems, one can see that the next iterations have the following form, for suitable $k_1,k_2<k+1$: 
	\begin{equation*}
	\UpLevVarsSingle^{k+1}_1=\frac{\UpLevVarsSingle_1^{k_1}+\UpLevVarsSingle_1^{k_2}}{2} \Komma \quad \quad \UpLevVarsSingle^{k+1}_2=\UpLevVarsSingle^{k_1}_1\cdot\UpLevVarsSingle_1^{k_2} \Punkt
	\end{equation*}
	It is well known that a bisection has only a linear rate of convergence, if it does not terminate after finitely many steps. For this example a finite termination is not possible because of a divisibility argument.	
\end{example}
The reason for the linear convergence in this example is due to the strict separation of lower-level and discretized problem. In every iteration in the Optimization Step a solution is found which maximizes the violation between two discretization points. To overcome this strict separation we will use the Reduction Ansatz.
\subsection{The Reduction Ansatz} \label{subsec:Reduction}
For a more detailed introduction see for example \cite{Hettich.1982,Jongen.1992,Klatte.1992}. For every $\UpLevIneqIndexForAll$ and $\UpLevVars^* \in \FeasSet$ we denote the set of \emph{active indices} by 
\begin{equation*}
\ActiveSet{\UpLevVars^*}:=\{\LowLevVars\in \Indexset\mid \UpLevIneqSingle(\UpLevVars^*,\LowLevVars)=0\} \Punkt
\end{equation*}
\def\LowLevIneqLagRed{\boldsymbol{\mu}^{*}}
\def\LowLevIneqLagSingleRed{\mu_{\LowLevIneqIndex}^{*}}
Every $\LowLevVars^*\in \ActiveSet{\UpLevVars}$ is a global solution to the lower level problem $\LowLevProblem{\UpLevIneqIndex}{\UpLevVars^*}$. If we assume that the Linear Independence Constraint Qualification (short: \LICQ) holds, then the global solution is a KKT point (see for example \cite{Bazaraa.2006}). This means that for every $\UpLevIneqIndexForAll$ and  $\LowLevVars^*\in \ActiveSet{\UpLevVars^*}$ there are $\LowLevIneqLagRed\in \LowLevIneqSpace$ satisfying the KKT conditions, i.e.
\begin{align}
D_2 \LowLevLag_{\UpLevIneqIndex}(\UpLevVars^*,\LowLevVars^*,\LowLevIneqLagRed)=&\ 0 \Komma\label{eq:KKT1}\\
\LowLevIneqLagSingleRed \cdot\LowLevIneqSingle(\LowLevVars^*) =&\ 0 \text { for all } \LowLevIneqIndexForAll \Komma\label{eq:KKT2}\\
\LowLevIneqLagSingleRed \geq&\  0  \text { for all } \LowLevIneqIndexForAll \Komma \label{eq:KKT3}
\end{align}
where, for every $\UpLevIneqIndexForAll$ the lower-level Lagrange function is denoted by
\begin{equation*}
\LowLevLag_{\UpLevIneqIndex}(\UpLevVars^*,\LowLevVars^*,\LowLevIneqLagRed)=\UpLevIneqSingle(\UpLevVars,\LowLevVars^*) -\sum_{\LowLevIneqIndexForAll}\LowLevIneqLagSingleRed\cdot  \LowLevIneqSingle(\LowLevVars^*)\Punkt
\end{equation*}
If in \eqref{eq:KKT2} exactly one of the two multipliers is equal to zero, then \emph{strict complementary slackness} is satisfied. Moreover, assuming that \LICQ\ and strict complementary slackness are satisfied, we say that the \emph{second-order sufficient condition is satisfied}, if the following holds:
\begin{equation*}
\SOSCDir^{\top} D_2^2 \LowLevLag_{\UpLevIneqIndex}(\UpLevVars^*,\LowLevVars^*,\LowLevIneqLagRed) \SOSCDir >0 \text{ for every } \SOSCDir \in \SOSCCone{\LowLevVars^*}, \SOSCDir \neq 0  \Komma
\end{equation*}
where
\begin{equation*}
\SOSCCone{\LowLevVars^*} = \left\{ \SOSCDir\in \NonLinSpace \left|
	D \LowLevIneqSingle(\LowLevVars^*)\SOSCDir= 0 \text{ for } \UpLevIneqIndex \in \UpLevIneqIndex \text{ with } \NonLinIneqLagSingle>0
\right.\right\} \Punkt
\end{equation*}
We can recall the following definition.
\begin{definition}
	The Reduction Ansatz holds in a point $\UpLevVars^*\in \FeasSet$, if for every $\UpLevIneqIndexForAll$ and $\LowLevVars^*\in \ActiveSet{\UpLevVars^*}$, \LICQ, strict complementary slackness and the second-order sufficient condition hold.
\end{definition}

It is well known, that under the Reduction Ansatz there are, for every $\UpLevIneqIndexForAll$, only finitely many active indices, i.e.
\begin{equation*}
\ActiveSet{\LimUpLev} = \{\LowLevMult\mid 1 \leq \LowLevMultIndex\leq \LowLevMultNum\} \Komma
\end{equation*}
where $\norm{\cdot}$ denotes the euclidean norm.
We denote the open ball around a vector $\UpLevVars^*\in \UpLevSpace$ with radius $\varepsilon>$ by
\begin{equation*}
\unitball{\varepsilon}{\UpLevVars^*} := \{\UpLevVars\in \UpLevSpace \mid \norm{\UpLevVars-\UpLevVars^*}<\varepsilon\}\Punkt
\end{equation*}
If the Reduction Ansatz holds, there are $\delta>0$, $\varepsilon>0$ and, for every $\UpLevIneqIndexForAll$ and $\LowLevMultIndex\in \{1,\dots, \LowLevMultNum\}$, differentiable functions 
\begin{align}
\LowLevMult:&\ \unitball{\delta}{\LimUpLev}\rightarrow \Indexset\Komma\label{eq:Red1}\\
\LowLevMultIneqLag:&\ \unitball{\delta}{\LimUpLev} \rightarrow \LowLevIneqSpace\label{eq:Red2}
\end{align}
such that, for every $\UpLevVars \in \unitball{\delta}{\UpLevVars^*}$, the point $\LowLevMult(\UpLevVars)$ is the unique stationary point within $\unitball{\varepsilon}{\LowLevVars}$ for the $\UpLevIneqIndex$-th lower-level problem $\LowLevProblem{\UpLevIneqIndex}{\UpLevVars}$ and $\LowLevIneqLag(\UpLevVars)$ are the unique Lagrange-multipliers satisfying the KKT conditions. If the describing function $\UpLevIneq$ and $\LowLevIneq$ are more then twice continuously differentiable, then the functions $\LowLevVars^{\UpLevIneqIndex}$ and $\LowLevIneqLag$ are more than once continuously differentiable.

The following local representation is true:
\begin{equation} \label{eq:RedAnsatz}
\FeasSet\cap\unitball{\delta}{\UpLevVars^*} = \{\UpLevVars\in \unitball{\delta}{\UpLevVars^*} \mid \forall\ \UpLevIneqIndexForAll, 1\leq k \leq \LowLevMultNum : \ \UpLevIneqSingle\big(\UpLevVars, \LowLevMult(\UpLevVars)\big)\leq 0 \}\Punkt
\end{equation}

For every $\UpLevIneqIndexForAll$ and $\LowLevMultIndex\in \{1,\dots \LowLevMultNum\}$ the derivative is given by:
\begin{equation*}
D_{\tilde{\UpLevVars}}\Big[\UpLevIneqSingle\big(\tilde{\UpLevVars},\LowLevMult(\tilde{\UpLevVars})\big)\Big]_{\tilde{\UpLevVars}=\UpLevVars}= D_1\UpLevIneqSingle(\UpLevVars,\LowLevMult(\UpLevVars))\Komma
\end{equation*}
which is again a differentiable function, which means that the description in \eqref{eq:RedAnsatz} is a local formulation of the problem by finitely many twice continuously differentiable functions.

\subsection{Approximate problems with linear information}

We now show how the Reduction Ansatz can be used to receive linear information about the lower-level problems. Using this information we develop a new adaptive discretization method. In the refinement step in the $\ItIndex$-th iteration we determine, for every $\UpLevIneqIndexForAll$, a solution $\ItLowLev$ of the lower-level problem $\LowLevProblem{\UpLevIneqIndex}{\ItUpLev}$.  

In a current iterate we collect all indices for which we can calculate a derivative of the lower-level solutions. We denote the set of indices by:
\begin{equation*}
\DiffIndices := \left\{\UpLevIneqIndexForAll\left| \begin{aligned}&
\ItLowLev \text{ satisfies \LICQ, strict complementary slackness and the second-}\\&\text{order sufficient condition for } \LowLevProblem{\UpLevIneqIndex}{\ItUpLev}
\end{aligned} \right.\right\}\Punkt
\end{equation*}
If $\lim_{\ItIndex \rightarrow \infty} \ItUpLev = \LimUpLev$ and the Reduction Ansatz holds at $\LimUpLev$, we have for sufficiently large $\ItIndex$
\begin{equation*}
\DiffIndices=\UpLevIneqIndexset \Punkt
\end{equation*}
For every $\UpLevIneqIndex\in \DiffIndices$ we have, as in Equations \eqref{eq:Red1} and \eqref{eq:Red2}, differentiable functions describing the locally unique stationary point and the corresponding Lagrange multipliers satisfying the KKT conditions. Using the implicit function theorem, we can calculate derivatives $D\LowLevVars^{\UpLevIneqIndex}(\ItUpLev)$ and $D\LowLevIneqLag(\ItUpLev)$. We can develop the functions linearly:
\begin{align*}
\ItLinLowLev(\UpLevVars):=&\ItLowLev+D\LowLevVars^{\UpLevIneqIndex}(\ItUpLev)\cdot (\UpLevVars-\ItUpLev)\Komma\\
\ItLinLowLevIneqLag(\UpLevVars):= & \ItLowLevIneqLag+D\LowLevIneqLag(\ItUpLev)\cdot(\UpLevVars-\ItUpLev)\Punkt
\end{align*}

In contrast to the discretized problem, we now consider a problem where we add linear information about the lower-level problem. Therefore, we use the lower-level Lagrange function and let, for $\UpLevIneqIndex\in \DiffIndices$:
\begin{equation*}
\ItLinConstr{\UpLevVars}:= \LowLevLag_{\UpLevIneqIndex}\big(\UpLevVars,\ItLinLowLev(\UpLevVars),\ItLinLowLevIneqLag(\UpLevVars)\big) \Punkt
\end{equation*}
One of the important properties of the new constraints is that, for every $\ItIndex \in \NNumbers$ and $\UpLevIneqIndexForAll$, the following holds analogous to the Reduction Ansatz:
\begin{equation}
D\ItLinConstr{\ItUpLev} = D_1 \UpLevIneqSingle(\ItUpLev,\ItLowLev)\label{eq:DerivLinConstr} \Punkt
\end{equation} 
To receive the next iterate $\ItUpLev[\ItIndex+1]$, we solve, using a discretization $\ItIndexset[\ItIndex+1]$, the following nonlinear optimization problem:
\begin{alignat*}{3}
\LinSIP{\ItIndexset[\ItIndex+1]}:\quad  &\min_{\UpLevVars\in \UpLevSpace} \ObjFun(\UpLevVars)\span \span\\
s.t. \quad && \UpLevIneqSingle(\UpLevVars,\LowLevVars)\leq&\ 0 \text{ for all }\UpLevIneqIndexForAll, \LowLevVars \in \ItIndexset[\ItIndex+1] \Komma\\
&& \ItLinConstr[\ItIndex]{\UpLevVars}\leq&\ 0\text{ for all } \UpLevIneqIndex\in \DiffIndices\Punkt
\end{alignat*}
With the above consideration we obtain the following algorithm:
\begin{algorithm}[H]
	\caption{Quadratic convergent adaptive discretization method (\algname)} \label{alg:ModificationSIP}
	\begin{algorithmic}[1]
		\STATE Input: initial point $\ItUpLev[0]$, initial discretization $\ItIndexset[0]\subseteq \Indexset, \ItIndex=0$.	
		\WHILE{termination criterion is not met}
		\FOR{$i \in I$}
		\STATE Compute a global solution  $\ItLowLev$ and Lagrange multipliers $\ItLowLevIneqLag$ of $\LowLevProblem{\UpLevIneqIndex}{\ItUpLev}$. \label{Step:ModNextDisc}
		\STATE Determine $D\LowLevVars^{\UpLevIneqIndex}(\ItUpLev)$ and $ D\LowLevIneqLag(\ItUpLev)$  if they exist.
		\ENDFOR
		\STATE $\ItIndexset[\ItIndex+1]=\ItIndexset \cup \bigcup_{\UpLevIneqIndexForAll}\{\ItLowLev\}$.
		\STATE Determine a solution $\ItUpLev[\ItIndex+1]$ of problem $\LinSIP[\ItIndex]{\ItIndexset[\ItIndex+1]}$. \label{Step:ModNextIt}	
		\STATE $\ItIndex=\ItIndex+1$.		
		\ENDWHILE
	\end{algorithmic}
\end{algorithm}

In the remainder of this paper two central convergence results are shown:
\begin{itemize}
	\item If every iterate $\ItUpLev$ is a stationary point of $\LinSIP{\ItIndexset[\ItIndex+1]}$, then any accumulation point is a stationary point of \SIP\ (Section \ref{sec:Stationary}).
	\item The iterates converge with quadratic rate of convergence (Section \ref{sec:Quadratic}).
\end{itemize}
In Section \ref{sec:Numeric} we will present a numerical example and show the improvements of the quadratic rate of convergence. Before we start to further investigate the convergence properties of Algorithm \algname, we first consider a slightly modified version of Example \ref{ex:LinearConvergence} introduced  at the beginning of this section. This example exhibits a linear rate of convergence, if the Blankenship and Falk algorithm is used:
\begin{example}\label{ex:LinearConvergence2}
	In Example \ref{ex:LinearConvergence} replace every $x_1$ by $x_1^2$. We obtain the following problem:
	\begin{align*}
	\SIPEx: \min \quad &-x_1^2+\frac{3}{2}x_2\\
	s.t. \quad&-\LowLevVarsSingle^2 +2\LowLevVarsSingle\cdot x_1^2-x_2 \leq 0 \quad \text{ for all } \enspace \LowLevVarsSingle\in[-1,1]\Komma\\
	&x_1 \in [0,1] ,x_2 \in [-1,1] \Punkt
	\end{align*}
	The solution of the lower-level problem is given by $y=x_1^2$, which is differentiable. The first iterates of the modified algorithm can be calculated numerically. They and their distance to the optimal solution are listed in Table \ref{tab:exLinearConvergence2}. As one can see, an accuracy of $10^{-6}$ is reached in $4$ iterations. With the bisection type iterates of the Blankenship and Falk algorithm, calculated in Example \ref{ex:LinearConvergence}, approximately $20$ iterations are needed to reach the same accuracy. One can clearly see that the convergence is much faster and the iterates seem to converge quadratically.
	\begin{table}[htbp]%[H]
		\centering
		\caption{First iterates for Example \ref{ex:LinearConvergence2} constructed according to Algorithm \algname}
		\begin{tabular}{ccc}
			\toprule
			Itaration $k$ & $\ItUpLev$ & $\norm{\ItUpLev-\LimUpLev}$
			\\
			\midrule$0$&$(1,-1)$&$1.1888$\\
			$1$&$(0,-1)$&$1.2522$\\
			$2$&$(0.707107,  0)$&$0.1708$\\
			$3$&$(0.573761,   0.108057)$&$0.0047$\\
			$4$&$(0.57735, 0.111111)$&$2.1979\cdot 10^{-07}$\\
			\bottomrule
		\end{tabular}
%		\caption{First iterates for Example \ref{ex:LinearConvergence2}}
		\label{tab:exLinearConvergence2}
	\end{table}

	Note that the substitution of $x_1$ by $x_1^2$ does not affect the rate of convergence for the Blankenship and Falk algorithm. We make this substitution as otherwise the algorithm with linear information would terminate after a single iteration.
\end{example}

\section{Convergence of stationary points}\label{sec:Stationary}
In this section we show that an accumulation point of the iterates is a stationary point of problem \SIP, if the iterates $\ItUpLev$ are stationary points (KKT points) of the approximate problem $\LinSIP[\ItIndex-1]{\ItIndexset}$.

 A feasible point, $\UpLevVars^* \in \FeasSet$ for which the Reduction Ansatz holds, is called \emph{stationary point} of \SIP, if there is a vector of multipliers $\UpLevIneqLag^*\geq 0$ such that
\begin{equation} \label{eq:KKTConditionLimit}
0=D \ObjFun(\LimUpLev) + \sum_{\UpLevIneqIndexForAll}\sum_{\LowLevVars\in \ActiveSet{\LimUpLev}}\LimUpLevIneqLagSingle_{\UpLevIneqIndex,\LowLevVars} D_1 \UpLevIneqSingle(\LimUpLev,\LowLevVars)\Punkt
\end{equation}
The real numbers $\UpLevIneqLagSingle^*_{\UpLevIneqIndex,\LowLevVars}$ are again called \emph{Lagrange multipliers}. Note that the cardinality of $\ActiveSet{\LimUpLev}$ is finite, if the Reduction Ansatz holds.

We will need two constraint qualifications which extend the classical Mangasarian-Fromovitz Constraint Qualification (short: \MFCQ) and \LICQ.
\begin{itemize}
	\item A feasible point, $\UpLevVars\in \FeasSet$, is said to satisfies the \emph{Extended Mangasarian-Fromovitz Constraint Qualification} (short: \EMFCQ), if there is a vector $\EMFCQVec \in \UpLevSpace$ with
	\begin{align*}
	D_1\UpLevIneqSingle(\UpLevVars,\LowLevVars)\EMFCQVec \leq&\ -1 \text{ for every } \UpLevIneqIndexForAll, \LowLevVars\in \ActiveSet{\UpLevVars} \Punkt
	\end{align*}	
	\item A feasible point, $\UpLevVars\in \FeasSet$, satisfies the \emph{Extended Linear Independence Constraint Qualification} (short: \ELICQ) for problem SIP, if the vectors:
	\begin{align*}
	D_1\UpLevIneqSingle(\UpLevVars,\LowLevVars),\quad  \UpLevIneqIndexForAll, \LowLevVars \in \ActiveSet{\UpLevVars}
	\end{align*}
	are linearly independent.	
\end{itemize}

It is well known  that if a point $\UpLevVars$ satisfies \ELICQ, it also satisfies \EMFCQ. Moreover, a local minimum $\UpLevVars$ which satisfies \EMFCQ\ is a stationary point (see for example \cite{Lopez.2007}).

To simplify the exposition of the following statement we make the following assumption.
\begin{assumption}\label{ass:KKTPoint}
	Let $\SeqUpLev$ be constructed according to Algorithm \algname. Assume that every $\ItUpLev$ is a KKT point of the approximate problem $\LinSIP[\ItIndex-1]{\ItIndexset}$ and there is an $\LimUpLev\in \UpLevSpace$ with
	\begin{equation*}
	\lim_{\ItIndex\rightarrow\infty} \ItUpLev=\LimUpLev \Punkt
	\end{equation*} 
	Assume that the Reduction Ansatz and holds at $\LimUpLev$.
	Further assume:
	\begin{itemize}
		\item For every $\UpLevIneqIndexForAll$, there is a $\LowLevVars\in \Indexset$ with
		\begin{equation*}
		\UpLevIneqSingle(\LimUpLev,\LowLevVars)=0 \Punkt
		\end{equation*}
		\item For every $\ItIndex\in\NNumbers$ the following holds:
		\begin{equation} \label{eq:AllIndexDiffbar}
		\DiffIndices = \UpLevIneqIndexset \Punkt
		\end{equation} 
		This means that all derivatives of the solutions of the lower-level exist.
	\end{itemize}
\end{assumption}
The assumption that the iterates converge is made to simplify the notation. For the following proof one can also choose a convergent subsequence. The last two assumptions are only made to avoid case distinctions in the proofs. If there is an $\UpLevIneqIndexForAll$ such that no active index exists, one can simply consider the problem locally without this semi-infinite constraint. As we assume the Reduction Ansatz to hold in the limit point $\LimUpLev$, there is a $\ItIndex'\in \NNumbers$ such that \eqref{eq:AllIndexDiffbar} holds for all $\ItIndex\geq \ItIndex'$. If this index is larger than $0$ we can start the algorithm again with the current point and the current discretization as initialization.

We first show that the limit point $\LimUpLev$ is feasible.
\begin{lemma}\label{lem:ModFeas}
	Let Assumption \ref{ass:KKTPoint} hold. The limit point $\LimUpLev$ is feasible. 
\end{lemma}
\begin{proof}
	By construction of the algorithm, for every $\ItIndex'\geq \ItIndex+1$ and $\UpLevIneqIndexForAll$, we have
	\begin{equation*}
	\UpLevIneqSingle(\ItUpLev[\ItIndex'],\ItLowLev)\leq 0 \Punkt
	\end{equation*}
	As $\Indexset$ is compact, we can, for every $\UpLevIneqIndexForAll$, choose a subsequence $\{\ItLowLev[\ItSubedIndex]\}_{\ItSubIndex\in \NNumbers}$  such that
	\begin{equation*}
	\lim_{\ItSubIndex \rightarrow \infty} \ItLowLev[\ItSubedIndex] = \LimLowLev \Punkt
	\end{equation*}  
	For an arbitrary point $\LowLevVars \in \Indexset $ and every $\UpLevIneqIndexForAll$, the following holds by continuity:
	\begin{align*}
	\UpLevIneqSingle(\LimUpLev,\LowLevVars)=&\ \lim_{\ItSubIndex\rightarrow \infty} \UpLevIneqSingle(\ItUpLev[\ItSubedIndex],\LowLevVars)\\
	\leq &\ \lim_{\ItSubIndex\rightarrow \infty} \UpLevIneqSingle(\ItUpLev[\ItSubedIndex],\ItLowLev[\ItSubedIndex])\\
	= &\ \UpLevIneqSingle(\lim_{\ItSubIndex\rightarrow \infty} \ItUpLev[\ItSubedIndex],\lim_{\ItSubIndex\rightarrow \infty }\ItLowLev[\ItSubedIndex]) \\
	= &\ \UpLevIneqSingle(\lim_{\ItSubIndex\rightarrow \infty} \ItUpLev[\ItIndex_{\ItSubIndex+1}],\lim_{\ItSubIndex\rightarrow \infty}\ItLowLev[\ItSubedIndex]) \\
	= &\ \lim_{\ItSubIndex\rightarrow \infty} \UpLevIneqSingle(\ItUpLev[\ItIndex_{\ItSubIndex+1}],\ItLowLev[\ItSubedIndex]) \leq 0 \Punkt
	\end{align*}
\end{proof}
It is noteworthy that exactly the same proof can be used to show the same property  for the original algorithm by Blankenship and Falk. The additional constraint is not needed. The feasibility of accumulation points is a global convergence property that is directly inherited.

To show that the limit point $\LimUpLev$ is a stationary point of \SIP\ we need to find multipliers satisfying Equation \eqref{eq:KKTConditionLimit}. The idea is to construct these multipliers as limits of the multipliers corresponding to the iterates. For every $\UpLevIneqIndexForAll$ we denote the set of active discretization points in the $\ItIndex$-th iteration by
\begin{equation*}
\ItIndexsetActive:= \{\LowLevVars\in \ItIndexset\mid \UpLevIneqSingle(\ItUpLev,\LowLevVars) = 0\} \Punkt
\end{equation*}
By assumption there exist, for every $\ItIndex\in \NNumbers$, Lagrange multipliers $\ItUpLevIneqLag\geq0$ and $\ItLinConstrLag\geq0$ such that:
\begin{equation}
0=D \ObjFun(\ItUpLev)+\sum_{\UpLevIneqIndexForAll}\sum_{\LowLevVars\in \ItIndexsetActive} \ItUpLevIneqLagSingle_{\UpLevIneqIndex,\LowLevVars} D_1 \UpLevIneqSingle(\ItUpLev,\LowLevVars) + \sum_{\UpLevIneqIndexForAll}\ItLinConstrLagSingle D \ItLinConstr[\ItIndex-1]{\ItUpLev}\Komma \label{eq:ItLagForKKT}
\end{equation}
where, for every $\UpLevIneqIndexForAll$, the multiplier $\ItLinConstrLagSingle$ is $0$ if $\ItLinConstr[\ItIndex-1]{\ItUpLev}<0$.

For the construction of the Lagrange-multipliers in \eqref{eq:KKTConditionLimit} we need to match the active indices and the additional constraints in a current iteration to the active indices in the limit. We do this with the next two lemmas, first for the active indices $\ItIndexsetActive(\ItUpLev)$, then for the additional constraints $\ItLinConstr[\ItIndex-1]{\ItUpLev}$. The proofs of these technical lemmas can be found in Appendix \ref{app:KKTPoints}.

\begin{lemma} \label{lem:ConvKKTBoundDisc}
	Let Assumption \ref{ass:KKTPoint} be satisfied. For every $\delta>0$ there is a $\ItIndex'\in \NNumbers$ such that, for all $\ItIndex\geq\ItIndex'$, $\UpLevIneqIndexForAll$ and $\LowLevVars \in \Indexset$ with
	\begin{equation*}
	\UpLevIneqSingle(\ItUpLev,\LowLevVars)=0 \Komma
	\end{equation*}
	there is a $\LimLowLev \in \ActiveSet{\LimUpLev}$ such that
	\begin{equation*}
	\norm{\LimLowLev-\LowLevVars}<\delta \Punkt
	\end{equation*}
\end{lemma}

By continuity and the lemma above we can, for every $\varepsilon>0$, choose a $\ItIndex'\in \NNumbers$ such that, for all $\ItIndex\geq\ItIndex'$, $\UpLevIneqIndexForAll$ and $\LowLevVars \in \ItIndexset$ with
\begin{equation*}
\UpLevIneqSingle(\ItUpLev,\LowLevVars)=0 \Komma
\end{equation*}
there is a $\LimLowLev \in \ActiveSet{\LimUpLev}$ such that
\begin{equation} \label{eq:ConvKKTBoundDeriv1}
\norm{D_1\UpLevIneqSingle(\LimUpLev,\LimLowLev)-D_1\UpLevIneqSingle(\ItUpLev,\LowLevVars)}<\varepsilon \Punkt
\end{equation}

Similarly we can also achieve a bound for the additional constraints:
\begin{lemma} \label{lem:ConvKKTBoundLinConstr}
	Let Assumption \ref{ass:KKTPoint} be satisfied. For $\UpLevIneqIndexForAll$ consider a converging subsequence $\{\ItLowLev[\ItSubedIndex]\}_{\ItSubIndex\in \NNumbers}$ with
	\begin{equation*}
	\LimLowLev:=\lim_{\ItSubIndex \rightarrow \infty} \ItLowLev[\ItSubedIndex]\Punkt
	\end{equation*}
	Then $\LimLowLev \in \ActiveSet{\LimUpLev}$. Moreover, for every $\varepsilon>0$, there is an $\ItSubIndex'\in \NNumbers$ such that, for every $\ItSubIndex>\ItSubIndex'$ and $\UpLevIneqIndexForAll$, the following holds:
	\begin{equation*}
	\norm{D_1 \UpLevIneqSingle(\LimUpLev,\LimLowLev)-D\ItLinConstr[\ItSubedIndex-1]{\ItUpLev[\ItSubedIndex]}}< \varepsilon\Punkt
	\end{equation*}
\end{lemma}

We can now construct multipliers which converge towards the Lagrange multipliers of $\LimUpLev$. Therefore, choose a $\delta>0$ such that, for every $\UpLevIneqIndexForAll$, the balls $\unitball{\delta}{\LowLevVars},\LowLevVars \in \ActiveSet{\LimUpLev}$, are disjoint. As $\ItUpLev$ is a stationary point, there are Lagrange-multipliers $\ItUpLevIneqLag\geq 0$ and $\ItLinConstrLag\geq 0$ as in \eqref{eq:ItLagForKKT}.

For every $\ItIndex\in \NNumbers, \UpLevIneqIndexForAll$ and $\LowLevVars \in \ActiveSet{\LimUpLev}$, let
\begin{equation}
\ItConstrLagSingle_{\UpLevIneqIndex,\LowLevVars}= \sum_{\substack{\dot{\LowLevVars} \in \ItIndexsetActive(\ItUpLev),\\ 
		\norm{\dot{\LowLevVars}-\LowLevVars}<\delta}} \ItUpLevIneqLagSingle_{\UpLevIneqIndex,\dot{\LowLevVars}} + \begin{cases}
\ItLinConstrLagSingle & \text{ if } \norm{\ItLowLev[\ItIndex-1]-\LowLevVars}<\delta\Komma\\
0 & \text{ otherwise }\Punkt
\end{cases} \label{eq:DefNewLag}
\end{equation}

We bound the constructed Lagrange multipliers in the following Lemma. A proof can be found in Appendix \ref{app:KKTPoints}.
\begin{lemma} \label{lem:KKTBounded}
	Let Assumption \ref{ass:KKTPoint} be satisfied. Assume that, for every $\ItIndex\in \NNumbers$, the current iterate $\ItUpLev$ is a stationary point of $\LinSIP[\ItIndex-1]{\ItIndexset}$ and that \EMFCQ\ is satisfied at $\LimUpLev$. There is a constant $\ConstantK>0$ such that, for sufficiently large $\ItIndex$, the following holds:
	\begin{equation*}
	\sum_{\UpLevIneqIndexForAll} \sum_{\LowLevVars\in \ActiveSet{\LimUpLev}} \ItConstrLagSingle_{\UpLevIneqIndex,\LowLevVars} \leq \ConstantK\Punkt
	\end{equation*}
\end{lemma}

We are ready to prove that the limit point is again a stationary point.
\begin{theorem} \label{th:ConvergenceKKT}
	Let Assumption \ref{ass:KKTPoint} be satisfied. Assume that, for every $\ItIndex\in \NNumbers$, the current iterate $\ItUpLev$ is a KKT point of $\LinSIP[\ItIndex-1]{\ItIndexset}$ and that \EMFCQ\ is satisfied at $\LimUpLev$. Then $\LimUpLev$ is a stationary point of $\SIP$.
\end{theorem}
\begin{proof}
	Choose an arbitrary $\varepsilon>0$ and a $\delta>0$ such that, for every $\UpLevIneqIndexForAll$, the balls $\unitball{\delta}{\LowLevVars},\LowLevVars \in \ActiveSet{\LimUpLev}$, are disjoint. As we know by Lemma \ref{lem:ModFeas} that the limit point $\LimUpLev$ is feasible, it remains to show that we have multipliers satisfying Equation \eqref{eq:KKTConditionLimit}.
	
	By Lemma \ref{lem:ConvKKTBoundDisc} we can choose a $\ItIndex_1\in \NNumbers$ such that, for every $\ItIndex\geq\ItIndex_1, \UpLevIneqIndexForAll$ and $\LowLevVars \in \ItIndexsetActive$, there is a $\LimLowLev \in \ActiveSet{\LimUpLev}$ with
	\begin{align*}
	\norm{\LowLevVars-\LimLowLev} <&\ \delta\Komma\\
	\norm{D_1 \UpLevIneqSingle(\ItUpLev,\LowLevVars) -D_1 \UpLevIneqSingle(\LimUpLev,\LimLowLev)} <&\ \frac{\varepsilon}{2\ConstantK}\Komma
	\end{align*}
	where $\ConstantK$ is chosen as in Lemma \ref{lem:KKTBounded}. By Lemma \ref{lem:ConvKKTBoundLinConstr} we can choose $\ItIndex_2 \in \NNumbers$ such that, for every  $\ItIndex\geq\ItIndex_2$ and $\UpLevIneqIndexForAll$, there is a $\LimLowLev \in \ActiveSet{\LimUpLev}$ with: 
	\begin{align}
	\norm{\ItLowLev[\ItIndex-1] - \LimLowLev}<&\ \delta \Komma\nonumber\\
	\norm{\ItLinConstr[\ItIndex-1]{\ItUpLev} - D_1\UpLevIneqSingle(\LimUpLev,\LimLowLev)}<&\ \frac{\varepsilon}{2\ConstantK} \Punkt\label{eq:LemConvKKTBoundLin}
	\end{align}
	Moreover, by the convergence of $\SeqUpLev$ and the continuous differentiability of the objective, there is a $\ItIndex_3\in \NNumbers$ such that, for $\ItIndex\geq\ItIndex_3$, the following holds:
	\begin{equation*}
	\norm{D\ObjFun(\LimUpLev)-D\ObjFun(\ItUpLev)}< \frac{\varepsilon}{2}\Punkt
	\end{equation*}
	
	Let $\ItIndex':= \max\{\ItIndex_1,\ItIndex_2,\ItIndex_3\}$. Combining, for $\ItIndex\geq\ItIndex'$, the three inequalities with the stationary condition for $\ItUpLev$ (see Equation \eqref{eq:ItLagForKKT}), the following holds:
	\begin{align}
	&\ \nonumber \normbig{D\ObjFun(\LimUpLev) +\sum_{\UpLevIneqIndexForAll}\sum_{\LowLevVars\in \ActiveSet{\LimUpLev}} \ItConstrLagSingle_{\UpLevIneqIndex,\LowLevVars} D \UpLevIneqSingle(\LimUpLev,\LowLevVars)}\\\nonumber 
	\leq&\ \norm{D\ObjFun(\LimUpLev)-D\ObjFun(\ItUpLev)}\\\nonumber  &\ + \sum_{\UpLevIneqIndexForAll}\sum_{\LowLevVars\in \ActiveSet{\LimUpLev}} \sum_{\substack{
			\dot{\LowLevVars}\in \ItIndexsetActive\\\nonumber  \norm{\dot{\LowLevVars}-\LowLevVars}<\delta
	}} \ItUpLevIneqLagSingle_{\UpLevIneqIndex,\dot{\LowLevVars}} \normbig{D_1\UpLevIneqSingle(\LimUpLev,\LowLevVars)-D_1\UpLevIneqSingle(\ItUpLev,\dot{\LowLevVars})}\\\nonumber 
	&\ +\sum_{\UpLevIneqIndexForAll}\ItLinConstrLagSingle \normbig{D_1\UpLevIneqSingle(\LimUpLev,\LimLowLev) - D \ItLinConstr[\ItIndex-1]{\ItUpLev}}\\\nonumber 
	<&\ \frac{\varepsilon}{2} +\frac{\varepsilon}{2\ConstantK} \cdot\left(\sum_{\UpLevIneqIndexForAll} \sum_{\LowLevVars\in \ActiveSet{\LimUpLev}} \ItConstrLagSingle_{\UpLevIneqIndex,\LowLevVars}\right)\\
	\leq &\ \varepsilon \Komma\label{eq:ConvKKTproofIneq}
	\end{align}
	where, for every $\UpLevIneqIndexForAll$ and $\ItIndex\geq \ItIndex$, the active index $\LimLowLev$ is chosen according to \eqref{eq:LemConvKKTBoundLin}. As, for every $\UpLevIneqIndexForAll$ and $\LowLevVars\in\ActiveSet{\LimUpLev}$, the multiplier $\ItConstrLagSingle_{\UpLevIneqIndex,\LowLevVars}$ are bounded by Lemma \ref{lem:KKTBounded}, we can choose a subsequence $\{\ItUpLev[\ItSubedIndex]\}_{\ItSubIndex\in \NNumbers}$ such that, for every $\UpLevIneqIndexForAll$ and $\LowLevVars\in \ActiveSet{\LimUpLev}$:
	\begin{equation*}
	\LimUpLevIneqLagSingle_{\UpLevIneqIndex,\LowLevVars} := \lim_{\ItSubIndex \rightarrow \infty}\ItConstrLagSingle[\ItSubedIndex]_{\UpLevIneqIndex,\LowLevVars}
	\end{equation*} 
	exists. From the inequality in Equation \eqref{eq:ConvKKTproofIneq} it follows:
	\begin{equation*}
	Df(\LimUpLev)+\sum_{\UpLevIneqIndexForAll}\sum_{\LowLevVars\in \ActiveSet{\LimUpLev}} \LimUpLevIneqLagSingle_{\UpLevIneqIndex,\LowLevVars} D_1\UpLevIneqSingle(\LimUpLev,\LowLevVars) = 0 \Punkt
	\end{equation*}
\end{proof}

%\begin{remark}
%	We presented the proof here in the context of Algorithm \ref{alg:ModificationSIP}, but the same proof can be used to show a similar statement for the Blankenship and Falk algorithm (Algorithm \ref{alg:BandF}). The only difference lies in the fact that no additional constraint has to be considered, which only simplifies the proof.
%\end{remark}

\section{Quadratic rate of convergence}\label{sec:Quadratic}
After we have shown that a limit point is a stationary point, we proof now that the iterates converge with a quadratic rate towards this limit point. One of the ideas of the proof is to interpret the current iterate as a stationary point of a permuted semi-infinite optimization problem. A possibility to bound the distance of stationary point of a permuted problem to a stationary point of the original problem is the so-called strong stability. This concept was originally introduced by Kojima in \cite{Kojima.1980} and Robinson in \cite{Robinson.1980} and extended to the case of semi-infinite problems by Rückmann in \cite{Ruckmann.1999}. In the first subsection we summarize the concept and formulate two corollaries which we need to proof the quadratic rate of convergence. In Subsection \ref{subsec:Quad} we then use the strong stability and present a proof for the quadratic rate of convergence.

\subsection{Strong stability of stationary points}

For a bounded open set, $\StabOpenSubset \subseteq \UpLevSpace$, and twice continuously differentiable functions, $\PermObjFun:\UpLevSpace \rightarrow \RNumbers$ and $\PermUpLevIneq:\UpLevSpace\times\LowLevSpace \rightarrow  \UpLevIneqSpace$, we measure the maximal deviation of these functions by
\begin{equation*}
\mathrm{norm}( \PermObjFun,\PermUpLevIneq,U) = \max\left\{ \begin{aligned}
&\sup_{\UpLevVars\in U}\max\big\{|\PermObjFun(\UpLevVars)|,\norm{D\PermObjFun(\UpLevVars)},\norm{D^2\PermObjFun(\UpLevVars)}\big\},\\
&\max_{\UpLevIneqIndexForAll}\sup_{\UpLevVars\in U}\max_{\LowLevVars\in \Indexset}\max\big\{|\PermUpLevIneqSingle(\UpLevVars,\LowLevVars)|,\norm{D\PermUpLevIneqSingle(\UpLevVars,\LowLevVars)},\norm{D^2\PermUpLevIneqSingle(\UpLevVars,\LowLevVars)}\big\}
\end{aligned}\right\}\Punkt
\end{equation*}
For $\delta > 0$ let
\begin{equation*}
\PermSet{\delta}{\StabOpenSubset}
:=\left\{ (\PermObjFun,\PermUpLevIneq)\left| \begin{matrix}
\PermObjFun: \UpLevSpace \rightarrow \RNumbers \text{ twice continuosly differentiable,}\\[0.2em]
\PermUpLevIneq: \UpLevSpace\times \LowLevSpace \rightarrow \UpLevIneqSpace \text{ twice continuously differentiable},\\[0.2em]
\mathrm{norm}(\PermObjFun,\PermUpLevIneq,U)<\delta
\end{matrix}  \right.\right\} \Punkt
\end{equation*}

We can define strong stability in the semi-infinite case.
\begin{definition} (Strong stability for \SIP\ \cite{Ruckmann.1999}) \label{def:SIPStrongStab}
	A stationary point $\UpLevVars^*$ of \SIP\ is called \emph{strongly stable}, if there is a $\varepsilon^*>0$ with the property that for every $\varepsilon\in (0,\varepsilon^*]$ there is a $\delta> 0$ such that, for every $(\PermObjFun,\PermUpLevIneq)\in \PermSet{\delta}{\unitball{\varepsilon^*}{\UpLevVars^*}}$,
	the problem 
	\begin{align*}
	\textsf{SIP}(\PermObjFun,\PermUpLevIneq): \min_{\UpLevVars\in \UpLevSpace} \quad& \ObjFun(\UpLevVars) +\PermObjFun(\UpLevVars)\\
	\text{s.t. } \quad& \UpLevIneqSingle(\UpLevVars,\LowLevVars)+\PermUpLevIneqSingle(\UpLevVars,\LowLevVars) \leq 0 \text{ for all } \UpLevIneqIndexForAll, \LowLevVars\in \Indexset
	\end{align*}
	has a within $\unitball{\varepsilon^*}{\UpLevVars^*}$ unique stationary point $\UpLevVars(\PermObjFun,\PermUpLevIneq)$ and 
	\begin{equation*}
	\norm{\UpLevVars^*-\UpLevVars(\PermObjFun,\PermUpLevIneq)}<\varepsilon \Punkt
	\end{equation*}
	We denote by $\FeasSet(\PermObjFun,\PermUpLevIneq)$ the feasible set of problem $\SIP(\PermObjFun,\PermUpLevIneq)$.
\end{definition}

We will need two properties which follow from the equivalent characterizations given by Rückmann in \cite{Ruckmann.1999}. We will present their proofs in Appendix \ref{app:Strong}. The first statement covers the case if \ELICQ\ is satisfied.
\begin{corollary} \label{cor:Strong1}
	Let $\UpLevVars^*$ be a strongly stable stationary point of \SIP\ with Lagrange multipliers $\UpLevIneqLag$. Assume \ELICQ\  is satisfied at $\UpLevVars^*$ and the Reduction Ansatz holds. Then there are $\delta>0 ,\varepsilon>0$ and an $L>0$ such that, for every $(\PermObjFun,\PermUpLevIneq) \in \PermSet{\delta}{\unitball{\varepsilon}{\UpLevVars}}$, the permuted problem, $\SIP(\PermObjFun,\PermUpLevIneq)$, has a stationary point, $\UpLevVars(\tilde{f},\tilde{g})$, with
	\begin{equation*}
	\norm{\UpLevVars^*-\UpLevVars(\PermObjFun,\PermUpLevIneq)}\leq L \cdot \mathrm{norm}\big(\tilde{f},\tilde{g},\unitball{\varepsilon}{\UpLevVars^*}\big) \Komma
	\end{equation*}
	which is unique within $\unitball{\varepsilon}{\UpLevVars^*}$.
\end{corollary} 
The second statement covers the case if \ELICQ\ is not satisfied but \EMFCQ\ is satisfied

\begin{corollary} \label{cor:Strong2}
	Let $\UpLevVars^*$ be a stationary point of \SIP. Suppose that the Reduction Ansatz and \EMFCQ\ are satisfied at $\UpLevVars^*$ and \ELICQ\ is not satisfied. Then $\UpLevVars^*$ is a strict local minimum of order two, i.e. there is a $\varepsilon>0$ and a $L>0$ such that, for every $\UpLevVars\in \unitball{\varepsilon}{\UpLevVars^*}$, the following holds:
	\begin{equation*}
	\ObjFun(\UpLevVars)-\ObjFun(\UpLevVars^*) \geq L \norm{\UpLevVars-\UpLevVars^*}^2
	\end{equation*}
\end{corollary}

\subsection{General regularity assumption and proof of quadratic rate of convergence}\label{subsec:Quad}

 We have already shown in Section \ref{sec:Stationary} that the limit of stationary points is again a stationary point. We therefore assume this in the following and show that the iterates converge with a quadratic rate. The main statement of this section is the Quadratic Convergence Theorem. We begin by strengthening our previous assumptions. To prove the quadratic rate of convergence, we assume the following:

\begin{assumptionGRA}
	Let the describing functions $\UpLevIneq$ and $\LowLevIneq$ be three times continuously differentiable and the objective $\ObjFun$ be twice continuously differentiable. Let $\SeqUpLev$ be constructed according to Algorithm \algname. Assume that, for every $\ItIndex\in \NNumbers$, the point $\ItUpLev$ is a stationary-point of $\LinSIP[\ItIndex-1]{\ItIndexset}$.  Assume there is a strongly stable stationary point $\LimUpLev \in \FeasSet$ with 
	\begin{equation*}
	\lim_{\ItIndex \rightarrow \infty} \ItUpLev =\LimUpLev \Punkt
	\end{equation*}
	Moreover assume that $\LimUpLev$ has the following properties:
	
	\begin{itemize}
		\item The constraint qualification \EMFCQ\ is satisfied.
		\item The Reduction Ansatz holds at $\LimUpLev$.
		\item For every $\UpLevIneqIndexForAll$, the global solution $\LimLowLev$ of the lower-level problem $\LowLevProblem{\UpLevIneqIndex}{\LimUpLev}$ is unique and
		\begin{equation}
		\UpLevIneqSingle(\LimUpLev,\LimLowLev) =0 \Punkt \label{eq:QuadConvNotNeededAssumption}
		\end{equation}
	\end{itemize}
	We finally assume that $\DiffIndices = \UpLevIneqIndexset$, for every $\ItIndex\in \NNumbers$.
\end{assumptionGRA}
The last assumption and Equation \eqref{eq:QuadConvNotNeededAssumption} are, similar to Assumption \ref{ass:KKTPoint}, only added to avoid case distinctions in the proof. While all remaining assumptions are basic regularity assumptions, the assumption of a unique solution of the lower-level problem is more restrictive. If there are, for one semi-infinite constraint, multiple active indices no quadratic rate can be expected for Algorithm \algname. In this case one has to modify the algorithm and determine in step \ref{Step:ModNextDisc} all local solutions instead of a single global solution. Then again a quadratic rate of convergence can be expected. 

In this section we proof the following statement showing a quadratic rate of convergence:
\begin{theoremQuad}
	Let the General Regularity Assumptions be satisfied. There is a constant $L$ such that, for sufficiently large $\ItIndex$, the following holds:
	\begin{equation*}
	\norm{\ItUpLev[\ItIndex+1]-\LimUpLev} \leq L \cdot \norm{\ItUpLev-\LimUpLev}^2 \Punkt
	\end{equation*}
\end{theoremQuad}

\noindent \textbf{\textit{Proof.} }As $\ItUpLev$ is a stationary point, we know that there are multipliers $\ItUpLevIneqLag\geq 0$ and $\ItLinConstrLag\geq 0$ such that
\begin{equation}\label{eq:QuadKKTOrig}
0= D \ObjFun(\ItUpLev[\ItIndex])+\sum_{\UpLevIneqIndexForAll}\sum_{\LowLevVars\in \ItIndexsetActive[\ItIndex]}\ItUpLevIneqLagSingle[\ItIndex]_{\UpLevIneqIndex,\LowLevVars} D_1 \UpLevIneqSingle(\ItUpLev[\ItIndex],\LowLevVars)+\sum_{\UpLevIneqIndexForAll}\ItLinConstrLagSingle[\ItIndex] D \ItLinConstr[\ItIndex-1]{\ItUpLev[\ItIndex]} \Punkt
\end{equation}
The main idea of the proof is to replace, for $\UpLevIneqIndexForAll$, every derivative of the constraints by $D_1 \UpLevIneqSingle(\ItUpLev[\ItIndex],\ItLowLev[\ItIndex])$. Equation \eqref{eq:QuadKKTOrig} will not hold with an equality anymore, but let for $\ItIndex\in \NNumbers$ and $\UpLevIneqIndexForAll$:
\begin{align}
\boldsymbol{\beta}^{\ItIndex}:=&\ D \ObjFun (\ItUpLev) + \sum_{\UpLevIneqIndexForAll} \ItConstrLagSingle[\ItIndex]_{\UpLevIneqIndex} D_1 \UpLevIneqSingle(\ItUpLev,\ItLowLev) \label{eq:DefDelta} \Komma\\
\alpha^{\ItIndex}_{\UpLevIneqIndex} :=&
\begin{cases}
\max\{0,\UpLevIneqSingle(\ItUpLev,\ItLowLev)\} 
& \text{ if } \ItConstrLagSingle_{\UpLevIneqIndex}=0 \Komma 
\\
\UpLevIneqSingle(\ItUpLev,\ItLowLev)& \text{ otherwise } \Komma
\end{cases}  \label{eq:DefEpsilon}
\end{align}
where we let for $\UpLevIneqIndexForAll$, as in Equation \eqref{eq:DefNewLag} in the previous section:
\begin{equation*}
\ItConstrLagSingle_{\UpLevIneqIndex} :=\ItLinConstrLagSingle+ \sum_{\LowLevVars\in \ItIndexsetActive}\ItUpLevIneqLagSingle_{\UpLevIneqIndex,\LowLevVars} \Punkt
\end{equation*}
With this definition the current iterate $\ItUpLev$ is a stationary point of the following permuted semi-infinite problem:
\begin{align*}
\SIP^{\ItIndex}_{\text{mod}}:\min_{\UpLevVars \in \UpLevSpace} \quad & \ObjFun(\UpLevVars) -\boldsymbol{\beta}^{\ItIndex+1}\cdot\UpLevVars\\
s.t.\quad & \UpLevIneqSingle(\UpLevVars,\LowLevVars) -\alpha^{\ItIndex+1}_{\UpLevIneqIndex} \leq 0 \text{ for all }\UpLevIneqIndexForAll, \LowLevVars\in \Indexset \Punkt\\
\end{align*}
The permutation can be controlled by the parameters $\boldsymbol{\alpha}^{\ItIndex}$ and $\boldsymbol{\beta}^{\ItIndex}$. In the next two lemmas we give bounds which will be needed to bound these parameters. Their proof can be found in Appendix \ref{app:Quad}.

\begin{lemma} \label{lem:ModQuadBoundsSpecial}
	Let the General Regularity Assumptions be satisfied.
	\begin{itemize}
		\item[i)] There is a $\ConstantOne\in \RNumbers$ such that, for sufficiently large $\ItIndex$ and $\UpLevIneqIndexForAll$, the following holds:
		\begin{equation*}
		|\UpLevIneqSingle(\ItUpLev,\ItLowLev)-\ItLinConstr[\ItIndex-1]{\ItUpLev}| \leq \ConstantOne\cdot \norm{\ItUpLev[\ItIndex]-\ItUpLev[\ItIndex-1]}^4 \Punkt
		\end{equation*}
		\item[ii)] There is a $\ConstantTwo\in \RNumbers$ such that, for sufficiently large $\ItIndex$ and $\UpLevIneqIndexForAll$, the following holds:
		\begin{equation*}
		\norm{D_1 \UpLevIneqSingle(\ItUpLev,\ItLowLev)-D\ItLinConstr[\ItIndex-1]{\ItUpLev}}\leq \ConstantTwo \norm{\ItUpLev-\ItUpLev[\ItIndex-1]}^2 \Punkt
		\end{equation*}
	\end{itemize}
\end{lemma}

If the only active constraints in the stationary conditions \eqref{eq:QuadKKTOrig} are the additional constraints $\ItLinConstr[\ItIndex-1]{\ItUpLev}$, then the inequalities derived in the above lemma would already suffice to bound the parameter $\boldsymbol{\beta}^{\ItIndex}$. But the points in $\ItIndexset$ are needed to guarantee convergence to a feasible point and we cannot exclude the possibility that some of the constraints 
\begin{equation*}
\UpLevIneqSingle(\ItUpLev,\LowLevVars) \leq 0 \text{ for all } \UpLevIneqIndexForAll, \LowLevVars\in \ItIndexset
\end{equation*}
are active. The following lemma investigates these constraints

\begin{lemma}\label{lem:QuadConvAcitveDiscBound}
	Let the General Regularity Assumptions be satisfied. There is a $\ConstantKThree$ such that, for sufficiently large $\ItIndex$, $\UpLevIneqIndexForAll$ and a $\LowLevVars \in \Indexset$ with $\UpLevIneqSingle(\ItUpLev,\LowLevVars) =0$, the following holds:
	\begin{equation*}
	\|\ItLowLev-\LowLevVars\| \leq \ConstantKThree\|\ItUpLev-\ItUpLev[\ItIndex-1]\|^2 \Komma
	\end{equation*}
\end{lemma}
As a direct consequence of this lemma, there is a constant $\ConstantKFour$ such that, for sufficiently large $\ItIndex$ and $\LowLevVars \in \ItIndexsetActive$, the following holds:
\begin{equation}
\norm{D_1\UpLevIneqSingle(\ItUpLev,\ItLowLev)-D_1\UpLevIneqSingle(\ItUpLev,\LowLevVars)} \leq \ConstantKFour \norm{\ItUpLev- \ItUpLev[\ItIndex-1]}^{2} \Punkt \label{eq:QuadConvDerivDisc}
\end{equation}

\begin{proof}[Continuation of proof of Quadratic Convergence Theorem] We can now bound the parameters $\boldsymbol{\alpha}^{\ItIndex}$ and $\boldsymbol{\beta}^{\ItIndex}$. By Equations \eqref{eq:QuadKKTOrig}, \eqref{eq:DefDelta} and the definition of $\ItConstrLagSingle_{\UpLevIneqIndex}$ the following holds:
	\begin{align}
	\norm{\boldsymbol{\beta}^{\ItIndex}}=&\ \norm{ D \ObjFun (\ItUpLev) + \sum_{\UpLevIneqIndexForAll} \ItConstrLagSingle_{\UpLevIneqIndex}  D_1 \UpLevIneqSingle(\ItUpLev,\ItLowLev)}\nonumber\\
	=&\ \norm{ \sum_{\UpLevIneqIndexForAll} \overline{\lambda}_{\UpLevIneqIndex}^{\ItIndex} \big(D_1 \UpLevIneqSingle(\ItUpLev,\ItLowLev)-D \ItLinConstr[\ItIndex-1]{\ItUpLev}\big)\nonumber\\ &\ +\sum_{\UpLevIneqIndexForAll}\sum_{\LowLevVars \in \ItIndexsetActive}\lambda_{\LowLevVars}^{\UpLevIneqIndex,\ItIndex} \big(D_1 \UpLevIneqSingle(\ItUpLev,\ItLowLev)-D_1 \UpLevIneqSingle(\ItUpLev,\LowLevVars)\big)   }\Punkt\nonumber\\
	\intertext{By Lemma \ref{lem:ModQuadBoundsSpecial} and the inequality given in Equation \eqref{eq:QuadConvDerivDisc}, the following holds :}	
		\norm{\boldsymbol{\beta}^{\ItIndex}}\leq&\  \left(\sum_{\UpLevIneqIndexForAll} \overline{\lambda}_{\UpLevIneqIndex}^{\ItIndex}+\sum_{\UpLevIneqIndexForAll}\sum_{\LowLevVars \in \ItIndexsetActive[\ItIndex]}\lambda_{\LowLevVars}^{\UpLevIneqIndex,\ItIndex}
	\right) \max\{\ConstantKTwo,\ConstantFour\} \norm{\ItUpLev[\ItIndex]-\ItUpLev[\ItIndex-1]}^2\Punkt\nonumber\\
	\intertext{By Lemma \ref{lem:KKTBounded} the Lagrange multipliers are bounded and there is a $\ConstantKFife>0$ such that}
	\norm{\boldsymbol{\beta}^{\ItIndex}}\leq&\  \ConstantKFife \norm{\ItUpLev-\ItUpLev[\ItIndex-1]}^2 \Punkt \label{eq:QuadConvProofIneq2}
	\end{align}
	To bound the absolute value of $\alpha^{k}_i$, for $\UpLevIneqIndexForAll$,  we have to make a case distinction.
	
	First case: assume that $\UpLevIneqSingle(\ItUpLev,\ItLowLev) \geq 0$. As $\ItLinConstr[\ItIndex]{\ItUpLev[\ItIndex+1]}\leq 0$, we have by Lemma \ref{lem:ModQuadBoundsSpecial}:
	\begin{equation}
	|\alpha^{k}_i| \leq \ConstantKTwo \norm{\ItUpLev-\ItUpLev[\ItIndex-1]}^4 \Punkt \label{eq:QuadConvIneqAlhpa1}
	\end{equation}
	
	Second case: assume that $\UpLevIneqSingle(\ItUpLev,\ItLowLev) < 0$. If $\ItConstrLagSingle_{\UpLevIneqIndex}=0$, then it holds $\alpha^{k}_i = 0$, by definition (see Equation \eqref{eq:DefEpsilon}). If $\ItConstrLagSingle_{\UpLevIneqIndex}>0$, then we must have $\overline{\lambda}_{\UpLevIneqIndex}^k>0$, as no $\LowLevVars\in \ItIndexset$ can be active. We have by the complementarity conditions:
	\begin{equation*}
	\ItLinConstr[\ItIndex-1]{\ItUpLev}=0 \Punkt
	\end{equation*}
	Again, Lemma \ref{lem:ModQuadBoundsSpecial} shows the same inequality  as in Equation \eqref{eq:QuadConvIneqAlhpa1}. 
	
	To use the two corollaries about strong stability, we need to make a case distinction
	
	\textbf{First Case: \ELICQ\ is satisfied at $\LimUpLev$.}	
	By the strong stability of $\LimUpLev$, we know using Corollary \ref{cor:Strong1}  that there is a constant $L'$ such that, for sufficiently large $\ItIndex$:
	\begin{equation}
	\norm{\ItUpLev[\ItIndex+1] -\LimUpLev} \leq L' \cdot \big(\norm{\boldsymbol{\alpha}^{\ItIndex+1}}+\norm{\boldsymbol{\beta}^{\ItIndex+1}}\big)\Punkt \label{eq:QuadConvProofIneq1}
	\end{equation}
	
	Using the inequality given in Equation \eqref{eq:QuadConvProofIneq1} in combination with the inequalities given in Equations \eqref{eq:QuadConvProofIneq2} and \eqref{eq:QuadConvIneqAlhpa1}, we receive further, for sufficiently large $\ItIndex$:
	\begin{align}
	\norm{\ItUpLev[\ItIndex+1]-\LimUpLev} \leq & L'(\ConstantKFife+\ConstantKTwo) \norm{\ItUpLev[\ItIndex+1]-\ItUpLev}^2 \nonumber\\
	\leq&L'(\ConstantKFife+\ConstantKTwo) \Big(\norm{\ItUpLev[\ItIndex+1]-\LimUpLev}^2+2\norm{\ItUpLev[\ItIndex+1]-\LimUpLev}\cdot \norm{\ItUpLev-\LimUpLev} +\norm{\ItUpLev-\LimUpLev}^2\Big) \nonumber\\
	\leq & \frac{1}{2} \norm{\ItUpLev[\ItIndex+1]-\LimUpLev} + L'(\ConstantKFife+\ConstantKTwo) \norm{\ItUpLev-\LimUpLev}^2 \Komma\label{eq:ProofQuadConvQuadConv}
	\end{align}	
	where the last inequality holds for sufficiently large $\ItIndex$, as the first two terms converge faster to $0$ than $ \norm{\ItUpLev[\ItIndex+1]-\LimUpLev}$. The claim is shown for $L=2 L'(L_1+\ConstantKTwo)$.
	
	\textbf{Second case: $\ELICQ$ does not hold at $\LimUpLev$.} We first introduce some neighbourhoods of $\LimUpLev$.
	
	(1) As \ELICQ\ does not hold at $\LimUpLev$, but \EMFCQ\ does, we know from Corollary \ref{cor:Strong2} that the limit point $\LimUpLev$ is a local minimum and there is an $\varepsilon_1>0$ and an $\ConstantL_1>0$ such that, for all $\UpLevVars \in \FeasSet \cap \unitball{\varepsilon_1}{\LimUpLev}$, the following holds:
	\begin{equation}
	\ObjFun(\UpLevVars)-\ObjFun(\LimUpLev)\geq \ConstantL_1 \cdot \norm{\UpLevVars-\LimUpLev}^2\Punkt \label{eq:ProofQuadConvEMFCQSecOrd}
	\end{equation}
	
	(2) We denote the feasible set of problem $\SIP^k_{\text{mod}}$ by $M^k$. By the strong stability there is an $\varepsilon_2>0$ such that, for sufficiently large $\ItIndex$, there is an unique stationary point of $\SIP^{\ItIndex}_{\text{mod}}$ within
	\begin{equation*}
	M^k \cap \unitball{\varepsilon_2}{\LimUpLev} \Punkt
	\end{equation*} 
	
	(3) As \EMFCQ\ holds at $\LimUpLev$ for \SIP, there is an $\EMFCQVec$ such that, for every $\UpLevIneqIndexForAll$, we have:
	\begin{equation*}
	D_1\UpLevIneqSingle(\LimUpLev,\LimLowLev)\EMFCQVec\leq -1 \Punkt
	\end{equation*}
	We can choose an $\varepsilon_3>0$ and a $\delta>0$ such that, for all $\UpLevVars\in \unitball{\varepsilon_3}{\LimUpLev}$, $\UpLevIneqIndexForAll$ and $\LowLevVars\in \unitball{\delta}{\LimLowLev}$, the following holds:
	\begin{equation}
	D_1\UpLevIneqSingle(\UpLevVars,\LowLevVars)\EMFCQVec \leq -\frac{1}{2} \Punkt \label{eq:ProofQuadCaseEMFCQ1}
	\end{equation}
	
	(4) For every $\UpLevIneqIndexForAll$, the set $Y^{i,\delta}:=\Indexset\setminus\unitball{\delta}{\LimLowLev}$ is compact. The maximum $\max_{\LowLevVars \in Y^{i,\delta}} \UpLevIneqSingle(\UpLevVars)$ is attained for every $\UpLevIneqIndexForAll$. By continuity we can choose an $\varepsilon_4$ such that, for every $\UpLevVars\in \unitball{\varepsilon_4}{\LimUpLev}$, $\UpLevIneqIndexForAll$ and $\LowLevVars\in Y^{i,\delta}$, the following holds:
	\begin{equation}
	\UpLevIneqSingle(\UpLevVars,\LowLevVars) \leq 0 \Punkt\label{eq:ProofQuadCaseEMFCQ2}
	\end{equation}
	
	Let $\varepsilon:= \min\{\varepsilon_1,\varepsilon_2,\varepsilon_3,\varepsilon_4\}$. Denote by $\hat{\UpLevVars}^{\ItIndex} \in M^k\cap \unitballClosed{\varepsilon/2}{\LimUpLev}$ a point with:
	\begin{equation*}
	\ObjFun(\hat{\UpLevVars}^{\ItIndex})=\min_{\UpLevVars \in M^k\cap \unitballClosed{\varepsilon/2}{\LimUpLev}} \ObjFun(\UpLevVars) \Komma
	\end{equation*}
	where $\unitballClosed{\varepsilon/2}{\LimUpLev}$ denotes the closure of $\unitball{\varepsilon/2}{\LimUpLev}$.	We will see in the following that $\ItUpLev = \hat{\UpLevVars}^{\ItIndex}$. We first move the constructed point towards feasibility. As $\max_{\UpLevIneqIndexForAll} \alpha_{\UpLevIneqIndex}^{\ItIndex}$ converges to $0$, we have, for sufficiently large $k$:
	\begin{equation}
	2 \norm{\EMFCQVec} \max_{\UpLevIneqIndexForAll} \alpha^k_i < \varepsilon/2 \Punkt
	\end{equation}
	By \eqref{eq:ProofQuadCaseEMFCQ2} and using a first-order Taylor expansion together with \eqref{eq:ProofQuadCaseEMFCQ1}, the point:
	\begin{equation*}
	\FeasIt:= \hat{\UpLevVars}^k + 2 \EMFCQVec \max_{\UpLevIneqIndexForAll} \alpha^k_i
	\end{equation*}
	is, for sufficiently large $\ItIndex$, feasible for \SIP. By continuity and the inequality given in \eqref{eq:QuadConvIneqAlhpa1} there is an $L_2>0$ such that, for sufficiently large $\ItIndex$:
	\begin{equation*}
	\ObjFun(\FeasIt)-\ObjFun(\hat{\UpLevVars}^{\ItIndex}) \leq L_2  \norm{\ItUpLev-\ItUpLev[\ItIndex-1]}^4 \Punkt
	\end{equation*}
	As by construction $\ObjFun(\hat{\UpLevVars}^{\ItIndex})\leq \ObjFun(\LimUpLev)$ we obtain using \eqref{eq:ProofQuadConvEMFCQSecOrd}, for  sufficiently large $\ItIndex$:
	\begin{equation}
	\norm{\hat{\UpLevVars}^{\ItIndex}- \LimUpLev} \leq \sqrt{\frac{L_2}{L_1}} \norm{\ItUpLev-\ItUpLev[\ItIndex-1]}^2 \Punkt \label{eq:ProofQuadSecCaseQuadBound}
	\end{equation}
	This means that, for sufficiently large $\ItIndex$, the point $\hat{\UpLevVars}^{\ItIndex}$ is contained in $\unitball{\varepsilon/2}{\LimUpLev}$ and is a local minimum of $\SIP_{\text{mod}}^{\ItIndex}$. From \eqref{eq:ProofQuadCaseEMFCQ1}, it follows that $\EMFCQ$ holds for $\SIP_{\text{mod}}^{\ItIndex}$. This means that $\hat{\UpLevVars}^{\ItIndex}$ is a stationary point. By uniqueness we have:
	\begin{equation*}
	\hat{\UpLevVars}^{\ItIndex}= \ItUpLev \Punkt
	\end{equation*}
	Using Equation \eqref{eq:ProofQuadSecCaseQuadBound}, the quadratic convergence now follows completely analogous to \eqref{eq:ProofQuadConvQuadConv}.
\end{proof}

\section{Numerical example}\label{sec:Numeric}
In the previous two sections we have shown that the limit point is a stationary point of the semi-infinite problem and that the iterates converge with a quadratic rate. We now present an easy numerical example which shows that the number of iterations can be reduced with the new Algorithm. We therefore compare the algorithm by Blankenship and Falk and the new Algorithm \algname.
A detailed numerical comparison using many examples would exceed the scope of this paper. The main focus of this work are the theoretical convergence properties shown in the two previous sections. The following example should only give an impression on how the new algorithm can speed up the convergence.

All implementations were done in MATLAB\textsuperscript{\textregistered} \cite{MATLAB.2016b}. All finite nonlinear problems were solved using the SQP method provided by the  MATLAB\textsuperscript{\textregistered} function \texttt{fmincon} (part of MATLAB\textsuperscript{\textregistered} Optimization Toolbox\texttrademark). The standard settings were used. All derivatives were calculated analytically and provided to \texttt{fmincon} as well as to the calculation of the linearization.

As a numerical example we consider the embedding of an ellipse into a triangle. We use the following description of an ellipse:
\begin{equation*}
D(\UpLevVars):=\Big\{\LowLevVars\in \RNumbers^2 \mid \big(\LowLevVars-\boldsymbol{c}(\UpLevVars)\big)^{\top}\left(A(\UpLevVars)A(\UpLevVars)^{\top}\right)^{-1}\big(\LowLevVars-\boldsymbol{c}(\UpLevVars)\big)\leq 1\Big\}\Komma
\end{equation*}
where
\begin{equation*}
\boldsymbol{c}(\UpLevVars)=\begin{pmatrix} x_1\\x_2\end{pmatrix} \quad \text{and} \quad A(\UpLevVars)=\begin{pmatrix}
x_3 & x_5\\
0 & x_4
\end{pmatrix} \Punkt
\end{equation*}

As a container we consider the following triangle
\begin{equation*}
C:=\left\{\LowLevVars\in \RNumbers^{2} \left| \begin{aligned}
g_1(\LowLevVars):=&&-y_1-1 \leq&\ 0,\\ 
g_2(\LowLevVars):=&&-y_2-1 \leq&\ 0,\\
g_3(\LowLevVars):=&&\frac{1}{4} y_1 +y_2 -\frac{3}{4} \leq&\ 0 
\end{aligned}\right.\right\}\Punkt
\end{equation*}
 
The design-centering problem consists of the maximization of of the volume $\mathrm{vol}(D(\UpLevVars))$ under the condition:
\begin{equation*}
D(\UpLevVars) \subseteq C
\end{equation*}
More details about design-centering problems can for example be found in \cite{Stein.2003} and \cite{Harwood.2017}. In this easy case the problem can be solved analytically and its solution is given by:
\begin{equation*}
\LimUpLev=\left(\frac{5}{3},-\frac{1}{3},\frac{4\sqrt{3}}{3},\frac{2}{3},-\frac{4}{3}\right)^{\top} \Punkt
\end{equation*}
The optimal solution is shown in Figure \ref{fig:Example1}.
\begin{figure}[H]
	\centering
	\includegraphics[width=0.6\textwidth]{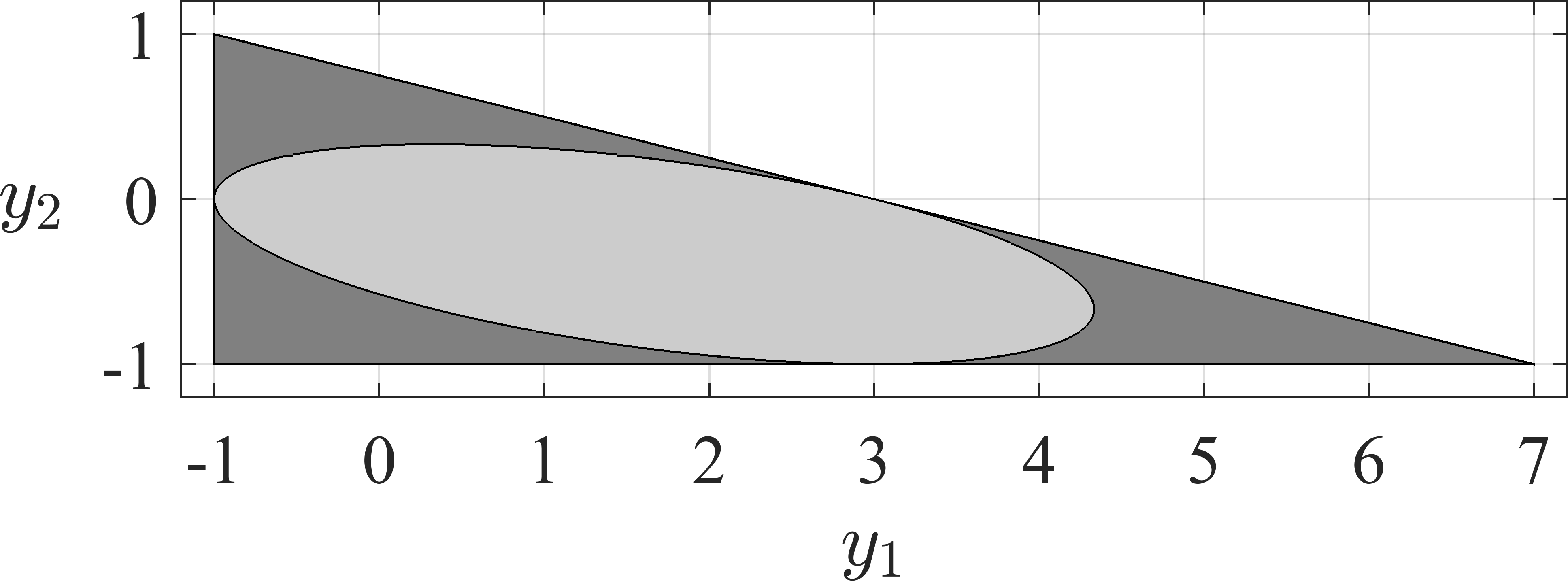}
	\caption{Largest ellipse in triangle $C$ [\emph{light gray} - ellipse $D(\LimUpLev)$, \emph{dark gray} - triangle $C$].}
	\label{fig:Example1}
\end{figure}
By letting $Y:=\{\LowLevVars \in \LowLevSpace \mid \LowLevVarsSingle_1^2+\LowLevVarsSingle_1^2 \leq 1\}$ and using $\TraffoFun(\UpLevVars,\LowLevVars)=A(\UpLevVars) \LowLevVars +\boldsymbol{c}(\UpLevVars)$ as a parametrization of the ellipse, the design-centering problem can be equivalently described by:
\begin{align*}
\SIP_{DC}: \max_{\UpLevVars\in \UpLevSpace}\quad & \pi \cdot x_3 \cdot x_4 \\
\text{ s.t.} \quad& \UpLevIneqSingle(\TraffoFun(\UpLevVars,\LowLevVars))\leq 0 \text{ for all } \UpLevIneqIndex\in \{1,2,3\}, \LowLevVars \in \Indexset \Punkt
\end{align*}
As initial point we used $\ItUpLev[1]=(0,0,1,1,0)^{\top}$. We terminated both algorithms the first time the distance to the optimal solution did fall below $10^{-4}$.

\def\algGSIP{\mathcal{A}_{\GSIP}}
\def\algSIP{\mathcal{A}_{1}}
\def\algBF{\mathcal{A}_{B\&F}}
The benefit of having access to the analytic solution is, that we can calculate the distance of the current iterates to a limit point. This distance for the first iterates is stated in Table \ref{tab:Example1}. The distance of iterates to the analytic solution are given in table \ref{tab:Example1}. While the distance is similar for the first iteration in further iterations the new algorithm \algname\ clearly outperforms the Blankenship and Falk algorithm ($\algBF$) for the iterations $2$ and $3$. In Iteration $3$ the termination criterion is reached for Algorithm $\algSIP$ which took $0.74$ seconds. The algorithm of Blankenship and Falk terminates after $14$ iteration which took $2.06$ seconds. 

\begin{table}[H]
	\centering
	\caption{Distances of the iterates to the optimal solution, $\norm{\ItUpLev-\LimUpLev}$, for the new algorithm with linear information (Algorithm \algname) and the original Blankenship and Falk algorithm ($\algBF$).}
	\begin{tabular}{cll}
		\toprule
		Iteration&$\algBF$ & \algname\\
		\midrule
		$0$&$2.5480$&$2.5480$\\
		$1$&$3.4975$&$2.156$\\
		$2$&$0.3374$&$3.9415\cdot 10^{-2}$\\
		$3$&$0.7586$&$1.0989\cdot 10^{-5}$\\
		$4$&$0.2543$&\multicolumn{1}{c}{-}\\
		$5$&$6.5553\cdot 10^{-2}$&\multicolumn{1}{c}{-}\\
		$6$&$2.3665\cdot 10^{-2}$&\multicolumn{1}{c}{-} \\
		\multicolumn{1}{c}{\vdots}&\multicolumn{1}{c}{\vdots} &\multicolumn{1}{c}{\vdots}\\
		$14$ &$4.9865\cdot 10^{-5}$ &\multicolumn{1}{c}{-}\\
		\bottomrule
	\end{tabular}
	\label{tab:Example1}
\end{table}
\section{Conclusion}
We started this paper with an example which shows that the classical discretization by Blankenship and Falk does in general not converge with a quadratic rate of convergence. This observation motivated a new adaptive discretization method. Using a linearization, we injected an additional constraint which accounts for the points not yet added to the discretization. We solve a new approximate problem on every iteration, having the constraints induced by the discretization points together with the injected constraint.

We showed two main convergence results. In Section \ref{sec:Stationary} we showed that iterates that are stationary points of the approximate problems, converge to a stationary point of the semi-infinite problem. Using strong stability of the limit point, we proved a quadratic rate of convergence under mild assumptions. In a numerical example we confirmed the quadratic rate and strongly reduced the number of iterations.

The smaller number of iterations also reduces the number of discretization points. As a result, the approximate problems have fewer constraints and are likely to be solved more quickly. Therefore, we are confident that the new algorithm will be faster than the classical adaptive discretization by Blankenship and Falk also in larger examples.

\bibliographystyle{abbrv}
\bibliography{LiteraturPaperQuadraticConvergence}

\appendix
\section{Proofs of auxiliary statements}
\subsection{Lemmas for convergence of stationary points} \label{app:KKTPoints}
\begin{proof}[Proof of Lemma \ref{lem:ConvKKTBoundDisc}]
	Fix throughout the proof an $\UpLevIneqIndexForAll$. The set
	\begin{equation*}
	\Indexset^{\UpLevIneqIndex,\delta}:=\Indexset\setminus \bigcup_{\LowLevVars\in \ActiveSet{\LimUpLev}} \unitball{\delta}{\LowLevVars}
	\end{equation*}
	is a compact set. Thus the maximum:
	\begin{equation*}
	\max \left\{\UpLevIneqSingle(\LimUpLev,\LowLevVars)\mid \LowLevVars\in \Indexset^{\UpLevIneqIndex,\delta} \right\}
	\end{equation*} is attained and is strictly less than $0$. Therefore, by continuity there is a $\ItIndex'$ such that, for $\ItIndex\geq\ItIndex'$ and $\LowLevVars\in \Indexset^{\UpLevIneqIndex,\delta}$:
	\begin{equation*}
	\UpLevIneqSingle(\ItUpLev,\LowLevVars)<0\Komma
	\end{equation*}
	which induces the claim.
\end{proof}
\begin{proof}[Proof of Lemma \ref{lem:ConvKKTBoundLinConstr}]
	For $\UpLevIneqIndexForAll$, consider an arbitrary converging subsequnce $\{\ItLowLev[\ItSubedIndex]\}_{\ItSubIndex\in \NNumbers}$. For $\LimLowLev$ and $\LowLevVars \in \Indexset$ the following holds:
	\begin{align*}
	\UpLevIneqSingle(\LimUpLev,\LimLowLev) =&\  \lim_{\ItSubIndex \rightarrow \infty} \UpLevIneqSingle(\ItUpLev[\ItSubedIndex],\ItLowLev[\ItSubedIndex])\\
	\geq &\  \lim_{\ItSubIndex \rightarrow \infty} \UpLevIneqSingle(\ItUpLev[\ItSubedIndex],\LowLevVars)\\
	= &\ \UpLevIneqSingle(\LimUpLev,\LowLevVars)\Punkt
	\end{align*}
	Which means $\LimLowLev$ is a global solution of $\LowLevProblem{\UpLevIneqIndex}{\LimUpLev}$. As we have assumed to have at least one active index, we must have $\LimLowLev \in \ActiveSet{\LimUpLev}$.
	
	We have seen in Equation \eqref{eq:DerivLinConstr} that, for every $\UpLevIneqIndexForAll$ and $\ItIndex\in \NNumbers$, the following holds:
	\begin{equation}\label{eq:LemKKTMatchLin1}
	D\ItLinConstr{\ItUpLev} = D_1 g(\ItUpLev,\ItLowLev) \Punkt
	\end{equation}
	
	For every $\varepsilon>0$ we can choose by continuity an $\ItSubIndex'\in \NNumbers$ such that, for every $\ItSubIndex>\ItSubIndex'$ and $\UpLevIneqIndexForAll$, we have:
	\begin{align*}
	\norm{D_1 \UpLevIneqSingle(\LimUpLev,\LimLowLev) - D_1 \UpLevIneqSingle(\ItUpLev[\ItSubedIndex-1],\ItLowLev[\ItSubedIndex-1])}<&\ \frac{\varepsilon}{2} \Komma\\
	\norm{D \ItLinConstr[\ItSubedIndex-1]{\ItUpLev[\ItSubedIndex-1]}-D \ItLinConstr[\ItSubedIndex-1]{\ItUpLev[\ItSubedIndex]}} <&\ \frac{\varepsilon}{2} \Punkt
	\end{align*}
	
	Combining these inequalities with Equation \eqref{eq:LemKKTMatchLin1} shows:
	\begin{align*}
	&\ \norm{D_1 \UpLevIneqSingle(\LimUpLev,\LimLowLev)-D \ItLinConstr[\ItSubedIndex-1]{\ItUpLev[\ItSubedIndex]}} \\
	\leq &\ \norm{D_1 \UpLevIneqSingle(\LimUpLev,\LimLowLev) -D_1 \UpLevIneqSingle(\ItUpLev[\ItSubedIndex-1],\ItLowLev[\ItSubedIndex-1])}\\
	&\ + \norm{D\ItLinConstr[\ItSubedIndex-1]{\ItUpLev[\ItSubedIndex-1]} -D\ItLinConstr[\ItSubedIndex-1]{\ItUpLev[\ItSubedIndex]}}\\
	<&\ \varepsilon \Punkt
	\end{align*}
\end{proof}

\begin{proof}[Proof of Lemma \ref{lem:KKTBounded}]
	As \EMFCQ\ holds, there is a $\EMFCQVec \in \UpLevSpace$ such that, for every $\UpLevIneqIndexForAll, \LowLevVars \in \ActiveSet{\LimUpLev}$:
	\begin{equation*}
	D_1 \UpLevIneqSingle (\LimUpLev,\LowLevVars) \EMFCQVec \leq -1\Punkt
	\end{equation*}
	By Continuity and Lemma \ref{lem:ConvKKTBoundLinConstr} we can choose a $\ItIndex'\in \NNumbers$ such that, for every $\ItIndex\geq\ItIndex', \UpLevIneqIndexForAll$ and $\LowLevVars \in \ItIndexsetActive$, the following is true:
	\begin{align*}
	D_1 \UpLevIneqSingle(\ItUpLev,\LowLevVars) \EMFCQVec\leq&\ -0.5 \Komma\\
	D \ItLinConstr[\ItIndex-1]{\ItUpLev}\EMFCQVec\leq &\ -0.5\Punkt
	\end{align*}
	Multiplying the stationarity condition given in \eqref{eq:ItLagForKKT} with vector $\EMFCQVec$ shows, for $\ItIndex\geq\ItIndex'$:
	\begin{align*}
	0=&\ \left(D \ObjFun(\ItUpLev)+\sum_{\UpLevIneqIndexForAll}\sum_{\LowLevVars\in \ItIndexsetActive} \ItUpLevIneqLagSingle_{\UpLevIneqIndex,\LowLevVars} D_1 \UpLevIneqSingle(\ItUpLev,\LowLevVars) + \sum_{\UpLevIneqIndexForAll}\ItLinConstrLagSingle D \ItLinConstr[\ItIndex-1]{\ItUpLev}\right)\cdot \EMFCQVec\\
	\leq&\ D \ObjFun(\ItUpLev)\EMFCQVec -0.5 \sum_{\UpLevIneqIndexForAll}\sum_{\LowLevVars\in \ItIndexsetActive} \ItUpLevIneqLagSingle_{\UpLevIneqIndex,\LowLevVars}  -0.5\sum_{\UpLevIneqIndexForAll}\ItLinConstrLagSingle \\%[-1em]
	%\noalign{\quad \quad by resorting and using the definition of the Lagrange multipliers:}
	=&\ D \ObjFun(\ItUpLev)\EMFCQVec -0.5 \sum_{\UpLevIneqIndexForAll}\sum_{\LowLevVars\in \ActiveSet{\LimUpLev}}\ItConstrLagSingle_{\UpLevIneqIndex,\LowLevVars} \Punkt 
	\end{align*}
	The boundedness now follows directly from the boundedness of $D\ObjFun(\ItUpLev)$.
\end{proof}
\subsection{Corollaries of strong stability} \label{app:Strong}
In \cite{Ruckmann.1999} for both cases (\ELICQ\ holds and \ELICQ\ does not hold, but \EMFCQ\ does) an equivalent characterization is shown (Theorem 2 and Theorem 3 in \cite{Ruckmann.1999}). The corollaries are easy consequences of these equivalent characterizations and their proofs.
\begin{proof}[Proof of Corollary \ref{cor:Strong1}]
	This is a direct consequence of the proof of Theorem 2 in \cite{Ruckmann.1999} which gives an equivalent characterization. In this proof the author first shows that the equivalent characterization induces the Lipschitz invertibility of a Function which characterizes stationary points. It is further shown that this invertibility induces a Lipschitz constant with the properties given in Corollary \ref{cor:Strong1}.  
\end{proof}
\begin{proof}[Proof of Corollary \ref{cor:Strong2}]
	For $\UpLevVars\in \UpLevSpace$ and $\UpLevIneqLag\geq 0$ let
	\begin{equation*}
	L(\UpLevVars, \UpLevIneqLag)= \ObjFun(\UpLevVars) +\sum_{\UpLevIneqIndexForAll}\sum_{\LowLevMultIndex=1}^{\LowLevMultNum} \UpLevIneqLagSingle_{\UpLevIneqIndex,\LowLevMultIndex}\cdot \UpLevIneqSingle\big(\UpLevVars,\LowLevMult(\UpLevVars)\big)\Komma
	\end{equation*}
	where $\LowLevMult(\UpLevVars), 1\leq \LowLevMultIndex \leq \LowLevMultNum$ are chosen as for the Reduction Ansatz in Section \ref{subsec:Reduction}. In \cite{Ruckmann.1999} the following equivalent characterization is shown:
	for every choice of Lagrange multipliers $\UpLevIneqLag\geq 0$ with
	\begin{equation*}
	D_1 L(\UpLevVars^*, \UpLevIneqLag) =0\Komma 
	\end{equation*}
	the following holds:
	\begin{equation*}
	\boldsymbol{d}^{\top} D_1^2 L(\UpLevVars,\UpLevIneqLag) \boldsymbol{d}>0 \text{ for all } \boldsymbol{d}\in R(\UpLevVars^*), \boldsymbol{d} \neq 0 \Komma
	\end{equation*}
	where 
	\begin{equation*}
	R(\UpLevVars^*):=\big\{\SOSCDir\in \UpLevSpace\mid D_1 \UpLevIneqSingle\big(\UpLevVars^*,\LowLevMult\big) \boldsymbol{d} = 0 \text{ for }  \UpLevIneqIndexForAll, 1 \leq \LowLevMultIndex\leq \LowLevMultNum \text{ with } \UpLevIneqLagSingle_{\UpLevIneqIndex,\LowLevMultIndex}>0\big\}\Punkt
	\end{equation*}
	This condition is stronger than the second-order sufficient condition given in \cite{Lopez.2007} (Theorem 5), which induces the claim.
\end{proof}
\subsection{Bounds for quadratic rate of convergence} \label{app:Quad}
\begin{proof}[Proof of Lemma \ref{lem:ModQuadBoundsSpecial}]
Consider throughout the proof a fixed index $\UpLevIneqIndexForAll$. By the Reduction Ansatz and the uniqueness of the active index $\LimLowLev$, there is a $\varepsilon>0$ and continuous differentiable functions
\begin{align*}
\LowLevVars^{\UpLevIneqIndex}:&\ \unitball{\varepsilon}{\LimUpLev} \rightarrow \UpLevSpace \Komma\\
\LowLevIneqLag:&\ \unitball{\varepsilon}{\LimUpLev} \rightarrow \UpLevSpace
\end{align*}
such that for every $\UpLevVars\in\unitball{\varepsilon}{\LimUpLev}$ the point $\LowLevVars^{\UpLevIneqIndex}(\UpLevVars)$ is the unique global solution of the lower level problem and $\LowLevIneqLag(\UpLevVars)$ are the unique corresponding Lagrange multipliers satisfying the KKT conditions. By the regularity assumptions these functions are twice continuously differentiable. For sufficiently large $\ItIndex$, we thus have $\ItLowLev = \LowLevVars^{\UpLevIneqIndex}(\ItUpLev)$. By a Taylor expansion of $\LowLevVars^{\UpLevIneqIndex}(\ItUpLev[\ItIndex])$ around $\ItUpLev[\ItIndex-1]$, there is an $\hat{\UpLevVars}= t \cdot \ItUpLev[\ItIndex]+(1-t)\cdot \ItUpLev[\ItIndex-1]$, for an appropriate $t\in [0,1]$,  such that
\begin{align*}
\norm{\ItLowLev[\ItIndex]-\ItLinLowLev[\ItIndex-1](\ItUpLev[\ItIndex])} =&\  \norm{\LowLevVars^{\UpLevIneqIndex}(\ItUpLev[\ItIndex])-\LowLevVars^{\UpLevIneqIndex}(\ItUpLev[\ItIndex-1])-D \LowLevVars^{\UpLevIneqIndex}(\ItUpLev[\ItIndex-1])(\ItUpLev[\ItIndex]-\ItUpLev[\ItIndex-1])}\\
=&\ \frac{1}{2}\norm{ (\ItUpLev[\ItIndex]-\ItUpLev[\ItIndex-1])^\top \cdot  D^2 \LowLevVars^{\UpLevIneqIndex}(\hat{\UpLevVars})\cdot(\ItUpLev[\ItIndex]-\ItUpLev[\ItIndex-1])} \Punkt
\end{align*}
Completely analogous the same holds for the Lagrange multipliers $\LowLevIneqLag$. By the continuity of the second derivative there are $\ConstantL_1\geq 0$ such that
\begin{align}
\norm{\ItLowLev[\ItIndex]-\ItLinLowLev[\ItIndex-1](\ItUpLev[\ItIndex])} \leq &\  \ConstantL_1 \norm{\ItUpLev[\ItIndex]-\ItUpLev[\ItIndex-1]}^2 \Komma \label{eq:QuadBound1}\\
\norm{\ItLowLevIneqLag[\ItIndex]-\ItLinLowLevIneqLag[\ItIndex-1](\ItUpLev[\ItIndex])} \leq &\  \ConstantL_1 \norm{\ItUpLev[\ItIndex]-\ItUpLev[\ItIndex-1]}^2 \Punkt\label{eq:QuadBound2}
\end{align}
We can now show both assertions:
\begin{itemize}
	\item[i)] Note that, for every $\ItIndex\in \NNumbers$, we have
	\begin{equation*}
	\UpLevIneqSingle(\ItUpLev,\ItLowLev) = \LowLevLag_{\UpLevIneqIndex}(\ItUpLev,\ItLowLev,\ItLowLevIneqLag) \Komma
	\end{equation*}
	by the complementarity conditions. By definition the following is also true:
	\begin{equation*}
	\ItLinConstr[\ItIndex-1]{\ItUpLev[\ItIndex]}=\LowLevLag_{\UpLevIneqIndex}\big(\ItUpLev,\ItLinLowLev[\ItIndex-1](\ItUpLev),\ItLinLowLevIneqLag[\ItIndex-1](\ItUpLev)\big) \Punkt
	\end{equation*}
	
	By the KKT conditions for the lower-level problem and strict complementary slackness, we have, for sufficiently large $\ItIndex$:
	\begin{align*}
	D_2 \LowLevLag_{\UpLevIneqIndex}(\ItUpLev,\ItLowLev,\ItLowLevIneqLag)\cdot(\ItLinLowLev[\ItIndex-1](\ItUpLev)-\ItLowLev)  =&\ 0 \Komma\\
	D_{3} \LowLevLag_{\UpLevIneqIndex}(\ItUpLev,\ItLowLev,\ItLowLevIneqLag)\cdot(\ItLinLowLevIneqLag[\ItIndex-1](\ItUpLev)-\ItLowLevIneqLag) =&\ 0 \Punkt
	\end{align*}
	
	This means that, for sufficiently large $\ItIndex$, by a Taylor expansion in $(\ItLowLev,\ItLowLevIneqLag)$, there are
	\begin{align*}
	\hat{\LowLevVars} =&\  t \cdot \ItLinLowLev[\ItIndex-1](\ItUpLev) +(1-t) \ItLowLev\Komma\\
	\hat{\boldsymbol{\mu}}=&\ t \cdot \ItLinLowLevIneqLag[\ItIndex-1](\ItUpLev)+(1-t)\ItLowLevIneqLag \Komma
	\end{align*}
	for an appropriate $t\in [0,1]$, such that:
	\begin{align*}
	&\norm{\UpLevIneqSingle(\ItUpLev,\ItLowLev)-\ItLinConstr[\ItIndex-1]\ItUpLev}\\
	=&\ \norm{\LowLevLag_{\UpLevIneqIndex}(\ItUpLev,\ItLowLev,\ItLowLevIneqLag)-\LowLevLag_{\UpLevIneqIndex}(\ItUpLev,\ItLinLowLev[\ItIndex-1](\ItUpLev),\ItLinLowLevIneqLag[\ItIndex-1](\ItUpLev))}\\
	=&\ \frac{1}{2}\norm{ \begin{pmatrix}\ItLinLowLev[\ItIndex-1](\ItUpLev)-\ItLowLev\\
		\ItLinLowLevIneqLag[\ItIndex-1](\ItUpLev)-\ItLowLevIneqLag
		\end{pmatrix}^\top\cdot  D^2_{2,3}  \LowLevLag_{\UpLevIneqIndex}(\ItUpLev, \hat{\LowLevVars},\hat{\boldsymbol{\mu}})\cdot 
		\begin{pmatrix}\ItLinLowLev[\ItIndex-1](\ItUpLev)-\ItLowLev\\
		\ItLinLowLevIneqLag[\ItIndex-1](\ItUpLev)-\ItLowLevIneqLag
		\end{pmatrix}}
	\end{align*}
	By continuity there is a $\Constant \in \mathbb{R}$ which bounds the second derivative and we have together with Equations \eqref{eq:QuadBound1} and \eqref{eq:QuadBound2}, for sufficiently large $\ItIndex$:
	\begin{align*}
	&\ \norm{\UpLevIneqSingle(\ItUpLev,\ItLowLev)-\ItLinConstr[\ItIndex-1]{\ItUpLev}}\\
	\leq &\ \frac{1}{2} \Constant\cdot \left(\norm{\ItLinLowLev[\ItIndex-1](\ItUpLev)-\ItLowLev}^2 +\norm{\ItLinLowLevIneqLag[\ItIndex-1](\ItUpLev)-\ItLowLevIneqLag}^2
	\right)\\
	\leq &\ \Constant\cdot \ConstantL_1\cdot \norm{\ItUpLev[\ItIndex+1]-\ItUpLev}^4 \Punkt
	\end{align*}
	\item[ii)]We first bound each term of the derivative and then combine these terms.
	
	Analogous to part i) we have, for sufficiently large $\ItIndex$:
	\begin{align*}
	D_2 \LowLevLag_{\UpLevIneqIndex}(\ItUpLev,\ItLowLev,\ItLowLevIneqLag)D\LowLevVars^{\UpLevIneqIndex}(\ItUpLev[\ItIndex-1]) = &\ 0 \Komma\\
	D_3 \LowLevLag_{\UpLevIneqIndex}(\ItUpLev,\ItLowLev,\ItLowLevIneqLag)D\LowLevIneqLag(\ItUpLev[\ItIndex-1]) = &\ 0 \Punkt
	\end{align*}
	As all involved functions are at least twice continuously differentiable, there is an $\ConstantL_2$ with
	\begin{align*}
	&\ \norm{D_{2,3}\LowLevLag_{\UpLevIneqIndex}(\ItUpLev,\ItLinLowLev[\ItIndex-1](\ItUpLev),\ItLinLowLevIneqLag[\ItIndex-1](\ItUpLev))\cdot\begin{pmatrix}
		D \LowLevVars^{\UpLevIneqIndex}(\ItUpLev[\ItIndex-1])\\
		D \LowLevIneqLag(\ItUpLev[\ItIndex-1])
		\end{pmatrix}}\\ \leq\ & L_2 (\norm{\ItLinLowLev[\ItIndex-1](\ItUpLev)-\ItLowLev}+\norm{\ItLinLowLevIneqLag[\ItIndex-1](\ItUpLev)-\ItLowLevIneqLag})\\ \leq\ & \ConstantL_2\cdot 2 \cdot \ConstantL_1 \cdot \norm{\ItUpLev-\ItUpLev[\ItIndex-1]}^2 \Komma
	\end{align*}
	where we used Equations \eqref{eq:QuadBound1} and \eqref{eq:QuadBound2} in the last inequality.
	Analogously there is also an $\ConstantL_3>0$ such that, for sufficiently large $\ItIndex$:
	\begin{align*}
	& \norm{D_{1}  \UpLevIneqSingle(\ItUpLev,\ItLowLev)-D_{1} \LowLevLag_{\UpLevIneqIndex}(\ItUpLev,\ItLinLowLev[\ItIndex-1](\ItUpLev),\ItLinLowLevIneqLag[\ItIndex-1](\ItUpLev))} \\
	=& \norm{D_{1} \UpLevIneqSingle(\ItUpLev,\ItLowLev)-D_{1}\UpLevIneqSingle(\ItUpLev,\ItLinLowLev[\ItIndex-1](\ItUpLev))}
	\\ \leq & \ConstantL_3 \cdot \ConstantL_1 \norm{\ItUpLev-\ItUpLev[\ItIndex-1]}^2
	\end{align*}
	We can now bound the derivative with the help of the chain-rule and the above inequalities, for sufficiently large $\ItIndex$:
	\begin{align*}
	&\norm{D_1 \UpLevIneqSingle(\ItUpLev,\ItLowLev)-D \ItLinConstr[\ItIndex-1]{\ItUpLev}} \\
	= &  \| D_{1}\UpLevIneqSingle(\ItUpLev,\ItLowLev) - D_{1}\LowLevLag_{\UpLevIneqIndex}(\ItUpLev,\ItLinLowLev[\ItIndex-1](\ItUpLev),\ItLinLowLevIneqLag[\ItIndex-1](\ItUpLev))\\
	&-D_{2,3}\LowLevLag_{\UpLevIneqIndex}(\ItUpLev,\ItLinLowLev[\ItIndex-1](\ItUpLev),\ItLinLowLevIneqLag[\ItIndex-1](\ItUpLev)) \cdot
	\begin{pmatrix}
	D \LowLevVars^{\UpLevIneqIndex}(\ItUpLev[\ItIndex-1])\\
	D \LowLevIneqLag(\ItUpLev[\ItIndex-1])
	\end{pmatrix} \|\\
	\leq &  \ConstantL_2\cdot 2 \cdot \ConstantL_1 \norm{\ItUpLev-\ItUpLev[\ItIndex-1]}^2 + \ConstantL_3 \cdot \ConstantL_1 \norm{\ItUpLev-\ItUpLev[\ItIndex-1]}^2 \Punkt
	\end{align*}
	Letting $\ConstantKTwo=(2\ConstantL_2+\ConstantL_3)\ConstantL_1$ the claim follows.
\end{itemize}
	
\end{proof}
\begin{proof}[Proof of Lemma \ref{lem:QuadConvAcitveDiscBound}]
	Again fix throughout the proof an $\UpLevIneqIndexForAll$.
	As we assumed the Reduction Ansatz to hold in $\LimUpLev$, the global solution $\LimLowLev$ is strongly stable as defined in \cite{Kojima.1980}. By exactly this property there is an $\varepsilon>0$, an $\ConstantL'>0$ and a $\ItIndex'\in\NNumbers$ such that, for every $0\leq\ConstantL\leq\ConstantL'$ and $\ItIndex\geq \ItIndex'$, there exists a unique stationary point of
	\begin{equation}
	\max_{\LowLevVars \in \Indexset} \UpLevIneqSingle(\ItUpLev,\LowLevVars)+\ConstantL \norm{\ItLowLev-\LowLevVars}^2 \label{eq:proofQuadConvLowGrowth}
	\end{equation}
	in $\unitball{\varepsilon}{\LimLowLev}$. For every $\ItIndex\in \NNumbers$, the KKT conditions are not affected at the point $\ItLowLev$ by the added term $\ConstantL \norm{\ItLowLev-\LowLevVars}^2$. This means that this point is still a stationary point. We need to show that it is still a local maximum.  By compactness the maximum
	\begin{equation*}
	\max_{\substack{\LowLevVars \in \Indexset\\ \norm{\LowLevVars -\LimLowLev}=\varepsilon}} \UpLevIneqSingle(\LimUpLev,\LowLevVars) -\UpLevIneqSingle(\LimUpLev,\LimLowLev)
	\end{equation*}
	is attained and is strictly less than $0$. We can therefore choose a $\ItIndex''$ and a $\ConstantL''>0$ such that, for every $\ItIndex\geq \ItIndex''$, $0\leq \ConstantL\leq \ConstantL''$ and $\LowLevVars \in \Indexset$ with $\norm{\LowLevVars-\LimLowLev}=\varepsilon$, the following holds:
	\begin{equation*}
	\UpLevIneqSingle(\ItUpLev,\LowLevVars)+ \ConstantL \norm{\ItLowLev-\LowLevVars}^2<\UpLevIneqSingle(\ItUpLev,\ItLowLev) \Punkt
	\end{equation*}
	As, for sufficiently large $\ItIndex$, the iterate of the lower-level $\ItLowLev$ lies within $\unitball{\varepsilon}{\LimLowLev}$, there must be a local minimum  within $\unitball{\varepsilon}{\LimLowLev} \cap \Indexset$ for the problem given in Equation \eqref{eq:proofQuadConvLowGrowth}. As \LICQ\ holds and by uniqueness of the stationary point, this point must be, for sufficiently large $\ItIndex$, the stationary point $\ItLowLev$. 
	
	For every $\LowLevVars \in \Indexset \cap \unitball{\varepsilon}{\LimLowLev}$ we then have for $\ConstantK:= \min\{\ConstantL',\ConstantL''\}$:
	\begin{equation}
	\UpLevIneqSingle(\ItUpLev,\LowLevVars) + \ConstantK \norm{\ItLowLev-\LowLevVars}^2 \leq \UpLevIneqSingle(\ItUpLev,\ItLowLev) \label{eq:QuadBoundLowerLevel}
	\end{equation}
	If now $\UpLevIneqSingle(\ItUpLev,\ItLowLev) <0$, there cannot be any active $\LowLevVars\in \Indexset$, because $\ItLowLev$ is the global maximum. For the remainder of the proof we therefore assume $\UpLevIneqSingle(\ItUpLev,\ItLowLev) \geq 0$. As $\ItUpLev$ is feasible for $\LinSIP[\ItIndex-1]{\ItIndexset}$, we know by Lemma \ref{lem:ModQuadBoundsSpecial}, for sufficiently large $\ItIndex$:
	\begin{equation*}
	\UpLevIneqSingle(\ItUpLev,\ItLowLev) \leq \ConstantOne \norm{\ItUpLev-\ItUpLev[\ItIndex-1]}^4 \Punkt
	\end{equation*}
	Combining this inequality with Equation \eqref{eq:QuadBoundLowerLevel} shows, for $\LowLevVars \in \Indexset$ with $\UpLevIneqSingle(\ItUpLev,\LowLevVars)=0$ and sufficiently large $\ItIndex$:
	\begin{equation*}
	\ConstantOne\norm{\ItUpLev-\ItUpLev[\ItIndex-1]}^4 \geq \UpLevIneqSingle(\ItUpLev,\ItLowLev)-\UpLevIneqSingle(\ItUpLev,\LowLevVars) \geq \ConstantK\norm{ \ItLowLev-\LowLevVars}^2 \Komma
	\end{equation*}	
	which proves the claim for all $\LowLevVars \in \Indexset$ with $\|\LowLevVars-\LimLowLev\|<\varepsilon$.
	
	Note that the set $\Indexset^{\UpLevIneqIndex}_{\varepsilon} := \{ \LowLevVars \in \Indexset \mid \|\LowLevVars-\LimLowLev\|\geq \varepsilon\}$ is compact. Thus the maximum $\max_{\LowLevVars \in \Indexset^{\UpLevIneqIndex}_{\varepsilon}} g(\LimUpLev,\LowLevVars)$ is attained and is strictly less then 0. By continuity we have, for sufficiently large $\ItIndex$ and for all $\LowLevVars\in \Indexset^{\UpLevIneqIndex}_{\varepsilon}$:
	\begin{equation*}
	\UpLevIneqSingle(\ItUpLev,\LowLevVars) <0 \Punkt
	\end{equation*} 
\end{proof}

\end{document}

%% file: commands.tex
\newcommand{\Punkt}{\ .}
\newcommand{\Komma}{\ ,}
\newcommand{\NNumbers}{\mathbb{N}}
\newcommand{\RNumbers}{\mathbb{R}}

\newcommand{\unitball}[2]{B_{#1}(#2)}
\newcommand{\unitballClosed}[2]{\overline{B}_{#1}(#2)}
\newcommand{\norm}[1]{\|#1\|}
\newcommand{\normbig}[1]{\big\|#1\big\|}

\newcommand{\ConstantL}{\ensuremath{L}}
\newcommand{\Constant}{K}

\newcommand{\ConstantOne}{K_1}
\newcommand{\ConstantTwo}{K_2}

\newcommand{\ConstantFour}{k_4}

\newcommand{\ConstantK}{K}

\newcommand{\ConstantKTwo}{K_2}
\newcommand{\ConstantKThree}{K_3}
\newcommand{\ConstantKFour}{K_4}
\newcommand{\ConstantKFife}{K_5}

\newcommand{\DimOptVar}{\ensuremath n}
\newcommand{\OptSpace}{\ensuremath{ \RNumbers^{\DimOptVar}}}
\newcommand{\FeasSet}{\ensuremath{M}}
\newcommand{\ObjFun}{\ensuremath{f}}

\newcommand{\NonLinSpace}{\OptSpace}

\newcommand{\NonLinIneqIndex}{\ensuremath{i}}

\newcommand{\NonLinIneqSingle}{\ensuremath{g_{\NonLinIneqIndex}}}

		\newcommand{\LICQ}{\textsf{\normalfont LICQ}}
		\newcommand{\MFCQ}{\textsf{\normalfont MFCQ}}

		\newcommand{\SOSCCone}[1]{T(#1)}
		\newcommand{\SOSCDir}{\ensuremath{\boldsymbol{d}}}

		\newcommand{\StabOpenSubset}{U}
		\newcommand{\PermObjFun}{\tilde{\ObjFun}}

		\newcommand{\PermSet}[2]{\mathcal{F}_{#1}(#2)}

\newcommand{\SIP}{\textup{\textsf{SIP}}}
\newcommand{\SIPEx}{\ensuremath{\SIP_{\text{ex}}}}
\newcommand{\Indexset}{Y}  
\newcommand{\UpLevVars}{\boldsymbol{x}}
\newcommand{\UpLevVarsSingle}{x}
\newcommand{\LowLevVars}{\boldsymbol{y}}
\newcommand{\LowLevVarsSingle}{y}
\newcommand{\UpLevDim}{n}
\newcommand{\LowLevDim}{m}
\newcommand{\UpLevSpace}{\RNumbers^{\UpLevDim}}
\newcommand{\LowLevSpace}{\RNumbers^{\LowLevDim}}

\newcommand{\UpLevIneqIndexset}{I}
\newcommand{\UpLevIneqIndex}{i}
\newcommand{\UpLevIneqIndexForAll}{\UpLevIneqIndex\in \UpLevIneqIndexset}
\newcommand{\UpLevIneq}{\boldsymbol{g}}
\newcommand{\UpLevIneqSingle}{g_{\UpLevIneqIndex}}
\newcommand{\UpLevIneqSpace}{\RNumbers^{|\UpLevIneqIndexset|}}

\newcommand{\LowLevIneqIndexset}{J}
\newcommand{\LowLevIneqIndex}{j}
\newcommand{\LowLevIneqIndexForAll}{\LowLevIneqIndex\in \LowLevIneqIndexset}
\newcommand{\LowLevIneq}{\boldsymbol{v}}
\newcommand{\LowLevIneqSingle}{v_{\LowLevIneqIndex}}
\newcommand{\LowLevIneqSpace}{\RNumbers^{|\LowLevIneqIndexset|}}

\newcommand{\LowLevProblem}[2]{Q_{#1}(#2)}

\newcommand{\LowLevLag}{\mathcal{L}}
\newcommand{\LowLevIneqLag}{\boldsymbol{\mu}^{\UpLevIneqIndex}}

\newcommand{\ActiveSet}[1]{\Indexset_0^{\UpLevIneqIndex}(#1)}

\newcommand{\FeasIt}{\bar{\UpLevVars}^{\ItIndex}}
	\newcommand{\LowLevMultIndex}{l}
	\newcommand{\LowLevMultNum}{q_i}
	\newcommand{\LowLevMult}[1][\LowLevMultIndex]{\LowLevVars^{\UpLevIneqIndex,#1}}
	\newcommand{\LowLevMultIneqLag}{\boldsymbol{\mu}^{\UpLevIneqIndex,\LowLevMultIndex}}

	\newcommand{\GSIP}{\textup{\textsf{GSIP}}}

	\newcommand{\TraffoFun}{\boldsymbol{t}}

	\newcommand{\UpLevIneqLag}{\boldsymbol{\lambda}}
	\newcommand{\UpLevIneqLagSingle}{\lambda}

	\newcommand{\NonLinIneqLagSingle}{\ensuremath{\lambda_{\NonLinIneqIndex}}}
	
	\newcommand{\ELICQ}{\textup{ELICQ}}
	\newcommand{\EMFCQ}{\textup{EMFCQ}}
	
	\newcommand{\EMFCQVec}{\boldsymbol{\xi}}
	
	\newcommand{\PermUpLevIneq}{\tilde{\UpLevIneq}}
	
	\newcommand{\PermUpLevIneqSingle}{\tilde{g}_{\UpLevIneqIndex}}

\newcommand{\DiscSIP}[1]{\SIP(#1)}

\newcommand{\ItIndex}{k}
\newcommand{\ItIndexset}[1][\ItIndex]{\Indexset^{#1}}
\newcommand{\ItLowLev}[1][\ItIndex]{\LowLevVars^{\UpLevIneqIndex,#1}}
\newcommand{\ItUpLev}[1][\ItIndex]{\UpLevVars^{#1}}
\newcommand{\LimUpLev}{\UpLevVars^*}
\newcommand{\LimLowLev}{\LowLevVars^{\NonLinIneqIndex,*}}
\newcommand{\ItLowLevIneqLag}[1][\ItIndex]{\boldsymbol{\mu}^{\UpLevIneqIndex,#1}}

\newcommand{\ItIndexsetActive}[1][\ItIndex]{\Indexset^{\UpLevIneqIndex,#1}_0}

\newcommand{\ItSubIndex}{l}
\newcommand{\ItSubedIndex}{\ItIndex_{\ItSubIndex}}

\newcommand{\SeqUpLev}{\{\ItUpLev\}_{\ItIndex \in \NNumbers}}

\newcommand{\LinSIP}[2][\ItIndex]{\overline{\textup{\textsf{SIP}}}^{#1}(#2)}

\newcommand{\ItLinConstr}[2][\ItIndex]{\overline{g}_{\UpLevIneqIndex}^{#1}(#2)}

\newcommand{\ItLinLowLev}[1][\ItIndex]{\overline{\LowLevVars}^{\UpLevIneqIndex,#1}}
\newcommand{\ItLinLowLevIneqLag}[1][\ItIndex]{\overline{\boldsymbol{\mu}}^{\UpLevIneqIndex,#1}}

\newcommand{\DiffIndices}[1][\ItIndex]{\overline{\UpLevIneqIndexset}^{#1}}

\newcommand{\LimUpLevIneqLagSingle}{\UpLevIneqLagSingle^{*}}
\newcommand{\ItUpLevIneqLag}[1][\ItIndex]{\UpLevIneqLag^{#1}}
\newcommand{\ItUpLevIneqLagSingle}[1][\ItIndex]{\UpLevIneqLagSingle^{#1}}
\newcommand{\ItLinConstrLag}[1][\ItIndex]{\overline{\UpLevIneqLag}^{#1}}
\newcommand{\ItLinConstrLagSingle}[1][\ItIndex]{\overline{\UpLevIneqLagSingle}^{#1}_{\UpLevIneqIndex}}

\newcommand{\ItConstrLagSingle}[1][\ItIndex]{\hat{\UpLevIneqLagSingle}^{#1}}

%% file: PaperQuadraticConvergence_Arxiv.bbl
\begin{thebibliography}{10}

\bibitem{Bazaraa.2006}
M.~S. Bazaraa, H.~D. Sherali, and C.~M. Shetty.
\newblock {\em {Nonlinear programming: Theory and algorithms / Mokhtar S.
  Bazaraa, Hanif D. Sherali, C.M. Shetty}}.
\newblock Wiley-Interscience and {Chichester : John Wiley [distributor]},
  Hoboken, N.J., 3rd ed. edition, 2006.

\bibitem{Blankenship.1976}
J.~W. Blankenship and J.~E. Falk.
\newblock {Infinitely constrained optimization problems}.
\newblock {\em {Journal of Optimization Theory and Applications}},
  19(2):261--281, 1976.

\bibitem{Diehl.2013}
M.~Diehl, B.~Houska, O.~Stein, and P.~Steuermann.
\newblock {A lifting method for generalized semi-infinite programs based on
  lower level Wolfe duality}.
\newblock {\em {Computational Optimization and Applications}}, 54(1):189--210,
  2013.

\bibitem{Djelassi.2019}
H.~Djelassi, M.~Glass, and A.~Mitsos.
\newblock {Discretization-based algorithms for generalized semi-infinite and
  bilevel programs with coupling equality constraints}.
\newblock {\em {Journal of Global Optimization}}, 75(2):341--392, 2019.

\bibitem{Harwood.2017}
S.~M. Harwood and P.~I. Barton.
\newblock {How to solve a design centering problem}.
\newblock {\em {Mathematical Methods of Operations Research}}, 86(1):215--254,
  2017.

\bibitem{Hettich.1993}
R.~Hettich and K.~O. Kortanek.
\newblock {Semi-infinite programming: theory, methods, and applications}.
\newblock {\em {SIAM Review}}, 35(3):380--429, 1993.

\bibitem{Hettich.1982}
R.~Hettich and P.~Zencke.
\newblock {\em {Numerische Methoden der Approximation und semi-infiniten
  Optimierung}}.
\newblock {Teubner Studienb{\"u}cher Mathematik}. {Vieweg+Teubner Verlag},
  Wiesbaden and s.l., 1982.

\bibitem{Jongen.1992}
H.~T. Jongen, F.~Twilt, and G.~W. Weber.
\newblock {Semi-infinite optimization: Structure and stability of the feasible
  set}.
\newblock {\em {Journal of Optimization Theory and Applications}},
  72(3):529--552, 1992.

\bibitem{Klatte.1992}
D.~Klatte.
\newblock {Stability of Stationary Solutions in Semi-Infinite Optimization via
  the Reduction Approach}.
\newblock In W.~Oettli and D.~Pallaschke, editors, {\em {Advances in
  optimization}}, volume 382 of {\em {Lecture Notes in Economics and
  Mathematical Systems}}, pages 155--170. Springer-Verlag, Berlin and New York,
  1992.

\bibitem{Kojima.1980}
M.~Kojima.
\newblock {Strongly Stable Stationary Solutions in Nonlinear Programs}.
\newblock In S.~M. Robinson, editor, {\em {Analysis and Computation of Fixed
  Points}}, pages 93--138. {Academic Press}, 1980.

\bibitem{Lopez.2007}
M.~L{\'o}pez and G.~Still.
\newblock {Semi-infinite programming}.
\newblock {\em {European Journal of Operational Research}}, 180(2):491--518,
  2007.

\bibitem{MATLAB.2016b}
{MATLAB\textsuperscript{\textregistered}, version 9.1}.
\newblock {(R2016b), The MathWorks Inc.}
\newblock 2016.

\bibitem{Mitsos.2011}
A.~Mitsos.
\newblock {Global optimization of semi-infinite programs via restriction of the
  right-hand side}.
\newblock {\em {Optimization}}, 60(10-11):1291--1308, 2011.

\bibitem{Mitsos.2015}
A.~Mitsos and A.~Tsoukalas.
\newblock {Global optimization of generalized semi-infinite programs via
  restriction of the right hand side}.
\newblock {\em {Journal of Global Optimization}}, 61(1), 2015.

\bibitem{Polak.1997}
E.~Polak.
\newblock {\em {Optimization: Algorithms and Consistent Approximations}},
  volume 124 of {\em {Applied Mathematical Sciences}}.
\newblock Springer-Verlag, New York, 1997.

\bibitem{Reemtsen.1991}
R.~Reemtsen.
\newblock {Discretization methods for the solution of semi-infinite programming
  problems}.
\newblock {\em {Journal of Optimization Theory and Applications}},
  71(1):85--103, 1991.

\bibitem{Reemtsen.1994}
R.~Reemtsen.
\newblock {Some outer approximation methods for semi-infinite optimization
  problems}.
\newblock {\em {Journal of Computational and Applied Mathematics}},
  53(1):87--108, 1994.

\bibitem{Reemtsen.1998b}
R.~Reemtsen and J.-J. R{\"u}ckmann, editors.
\newblock {\em {Semi-Infinite Programming}}.
\newblock {Springer US}, Boston, MA, 1998.

\bibitem{Robinson.1980}
S.~M. Robinson.
\newblock {Strongly Regular Generalized Equations}.
\newblock {\em {Mathematics of Operations Research}}, 5(1):43--62, 1980.

\bibitem{Ruckmann.1999}
J.-J. R{\"u}ckmann.
\newblock {On existence and uniqueness of stationary points in semi-infinite
  optimization}.
\newblock {\em {Mathematical Programming, Series B}}, 86(2):387--415, 1999.

\bibitem{Schwientek.2013}
J.~Schwientek.
\newblock {\em {Modellierung und L{\"o}sung parametrischer Packungsprobleme
  mittels semi-infiniter Optimierung: Angewandt auf die Verwertung von
  Edelsteinen: Zugl.: Kaiserslautern, Techn. Univ., Diss., 2013}}.
\newblock Fraunhofer-Verl., Stuttgart, 2013.

\bibitem{Stein.2003}
O.~Stein.
\newblock {\em {Bi-Level Strategies in Semi-Infinite Programming}}, volume~71
  of {\em {Nonconvex Optimization and Its Applications}}.
\newblock {Springer US}, Boston, MA and s.l., 2003.

\bibitem{Stein.2003b}
O.~Stein and G.~Still.
\newblock {Solving Semi-Infinite Optimization Problems with Interior Point
  Techniques}.
\newblock {\em {SIAM Journal on Control and Optimization}}, 42(3):769--788,
  2003.

\bibitem{Stein.2010}
O.~Stein and A.~Winterfeld.
\newblock {Feasible Method for Generalized Semi-Infinite Programming}.
\newblock {\em {Journal of Optimization Theory and Applications}},
  146(2):419--443, 2010.

\bibitem{Still.2001b}
G.~Still.
\newblock {Discretization in semi-infinite programming: the rate of
  convergence}.
\newblock {\em {Mathematical Programming}}, 91(1):53--69, 2001.

\bibitem{Tsoukalas.2011}
A.~Tsoukalas and B.~Rustem.
\newblock {A feasible point adaptation of the Blankenship and Falk algorithm
  for semi-infinite programming}.
\newblock {\em {Optimization Letters}}, 5(4):705--716, 2011.

\end{thebibliography}
